\documentclass[11pt]{amsart}
\usepackage{mathptmx}
\usepackage{amssymb}
\usepackage{enumerate}

\catcode`\@=11

\long\def\@savemarbox#1#2{\global\setbox#1\vtop{\hsize\marginparwidth 
  \@parboxrestore\tiny\raggedright #2}}
\newcommand\lref[1]{\ref{#1}%
\@ifundefined{r@DisplaY #1}{}{ (#1)}}

\newcommand\fakelabel[2]{\@bsphack\if@filesw {\let\thepage\relax
   \newcommand\protect{\noexpand\noexpand\noexpand}%
\xdef\@gtempa{\write\@auxout{\string
      \newlabel{#1}{{#2}{\thepage}}}}}\@gtempa
   \if@nobreak \ifvmode\nobreak\fi\fi\fi\@esphack}

\def\SL@margintext#1{{\showlabelsetlabel{\tiny\{\SL@prlabelname{#1}\}}}}
\catcode`\@=12

\def\Empty{}
\newcommand\oplabel[1]{
  \def\OpArg{#1} \ifx \OpArg\Empty {} \else
        \label{#1}
  \fi}
\newtheorem{theoremSt}{Theorem}[section]

\newtheorem{exampleSt}[theoremSt]{Example}
\newtheorem{exerciseSt}[theoremSt]{Exercise}

\newcommand\MakeStEnv[1]{
  \newenvironment{#1}[1]{
  \begin{#1St} \oplabel{##1}%
  \global\def\CrntSt{\thetheoremSt}%
}{ 
  \end{#1St} }
  \newenvironment{#1+}[1]{
  \begin{#1St} \label{##1}%
  \label{DisplaY ##1}%
  \global\def\CrntSt{\thetheoremSt}%
  \def\Labl{##1}\ifx\Labl\Empty{} \else {\em (\Labl)\,}\fi%
}{ 
  \end{#1St} }
}
\MakeStEnv{theorem}
\MakeStEnv{corollary}
\MakeStEnv{proposition}
\MakeStEnv{lemma}
\MakeStEnv{definition}
\MakeStEnv{conjecture}
\MakeStEnv{problem}
\MakeStEnv{question}

\newenvironment{example}[1]{
  \begin{exampleSt} \oplabel{#1}%
  \global\def\CrntSt{\thetheoremSt}%
  \rm %
}{ 
  \end{exampleSt} }

\newlength{\saveu}

\newenvironment{pf*}[1]{%
 \begin{proof}[#1]%
}{ 
 \end{proof}
}

\newcommand{\finishproof}[1]{ 
  \def\FPArg{#1}
  \ifx\FPArg\Empty
        \newcommand\FPArg{\CrntSt}  \fi
  \smallbreak\noindent\makebox[\textwidth]{\hfill\fbox{\FPArg}}
  \medbreak\noindent
}

\newcommand\CC{{\mathcal C}}

\newcommand\FF{{\mathcal F}}

\newcommand\LL{{\mathcal L}}
\newcommand\MM{{\mathcal M}}

\newcommand\PP{{\mathcal P}}

\newcommand\PMF{{\PP\kern-2pt\MM\FF}}

\newcommand\PML{{\PP\kern-2pt\MM\LL}}

\newcommand\ep{\epsilon}

\newcommand\bbR{{\mathord{\text{I\kern-2pt R}}}}        
\newcommand\bbH{{\mathord{\text{I\kern-2pt H}}}}        

\newcommand\C{{\mathbb C}}

\newcommand\Z{{\mathbb Z}}
\newcommand\R{{\mathbb R}}
\newcommand\N{{\mathbb N}}

\newcommand\bigrightarrow[1]{\hbox to #1{\rightarrowfill}}
\newcommand\bigleftarrow[1]{\hbox to #1{\leftarrowfill}}

\newcommand\semidir{\mathrel{\hbox{\vrule depth-.03ex height1.1ex\kern-0.15em$\times$}}}

\numberwithin{equation}{section}

\newcommand{\collar}{\operatorname{\mathbf{collar}}}

\newcommand{\fsubd}{\mathrel{{\scriptstyle\searrow}\kern-1ex^d\kern0.5ex}}
\newcommand{\bsubd}{\mathrel{{\scriptstyle\swarrow}\kern-1.6ex^d\kern0.8ex}}
\newcommand{\fsubeq}{\mathrel{\raise-.7ex\hbox{$\overset{\searrow}{=}$}}}
\newcommand{\bsubeq}{\mathrel{\raise-.7ex\hbox{$\overset{\swarrow}{=}$}}}

\newcommand{\tsh}[1]{\left\{\kern-.9ex\left\{#1\right\}\kern-.9ex\right\}}

\newcommand\MT{{\mathbb T}}

\def\IH{{\Bbb H}}
\def\Ht{{\IH^3}}

\def\IZ{{\Bbb Z}}
\def\IN{{\Bbb N}}

\begin{document}
\title{Convergence and divergence of Kleinian surface groups}
\author[Brock, Bromberg, Canary and Lecuire]{Jeffrey  Brock, Kenneth Bromberg, Richard Canary and Cyril Lecuire}
\thanks{Brock was partially supported by NSF grant DMS-1207572, Bromberg was partially supported by NSF grant 
DMS-1207873,
Canary  was partially
supported by NSF grants DMS-1006298 and DMS -1306992, and Lecuire was partially supported by the
ANR grant GDSOUS and  
Canary and Lecuire were partially supported by the GEAR network
(NSF grants DMS-1107452, 1107263, and 1107367)}

\begin{abstract}
  We characterize sequences of Kleinian surface groups with convergent
  subsequences in terms of the asymptotic behavior of the ending
  invariants of the associated hyperbolic 3-manifolds.  Asymptotic
  behavior of end invariants in a convergent sequence predicts the
  parabolic locus of the algebraic limit as well as how the algebraic
  limit wraps within the geometric limit under the natural locally
  isometric covering map.
\end{abstract}

\maketitle

\section{Introduction}
Central to Thurston's original approach to the hyperbolization theorem
for closed, irreducible, atoroidal 3-manifolds is a collection of
compactness criteria for deformation spaces of hyperbolic
3-manifolds. In the Haken setting, such compactness results gave rise
to iterative solutions to the search for hyperbolic structures on
constituent pieces in a hierarchical decomposition.

Later, the classification of hyperbolic 3-manifolds with finitely
generated fundamental group gave explicit {\em a priori} geometric
control of these manifolds in terms of the combinatorics of the
asymptotic data determining the hyperbolic structure, up to
bi-Lipschitz diffeomorphism.  Sullivan's Rigidity Theorem then allows
for the passage from bi-Lipschitz diffeomorphism to isometry.  The
invariants themselves then become parameters, and the bi-Lipschitz
control they provide gives rise to a new range of interrelations
between geometric and topological features of the resulting manifolds.

The present paper relates these asymptotic invariants explicitly to
compactness criteria, characterizing subsequential convergence
precisely in terms of the invariants' limiting combinatorics vis a vis
the {\em complex of curves}.  In particular, we describe a manner in
which invariants {\em bound projections} to curve complexes of
subsurfaces, a notion that guarantees {\em a priori} bounds for
geodesic lengths in a sequence. Our main theorem is a generalization
of Thurston's Double Limit Theorem (\cite{thurston2,otal-double}),
which provides a criterion to ensure subsequential convergence of a
sequence of Kleinian surface groups, and is a key technical step in
Thurston's hyperbolization theorem for 3-manifolds fibering over the
circle.

\begin{theorem}{main theorem}
Let $S$ be a compact, orientable surface and let $\{\rho_n\}$ be a
sequence in $AH(S)$ with end invariants $\{\nu_n^\pm\}$. Then
$\{\rho_n\}$ has a convergent subsequence if and only if 
there exists a subsequence $\{\rho_j\}$ of $\{\rho_n\}$ such
that $\{\nu_j^\pm\}$ bounds projections.
\end{theorem}

We also (see Theorem \ref{predictive power}) show that the asymptotic
behavior of the end invariants predicts the curve and lamination
components of the end invariants of the limit and how the algebraic
limit manifold ``wraps'' within a geometric limit.

We briefly describe terms and notation of Theorem~\ref{main
  theorem}.

Recall that $AH(S)$ is the space of (conjugacy classes of)
representations $$\rho:\pi_1(S)\to {\rm PSL}(2,\mathbb C)$$ for which
$\rho$ sends peripheral elements to parabolic elements. The {\em end
  invariants} will be discussed more thoroughly in
Section~\ref{background}, but in the case that $\rho$ is {\em
  quasi-Fuchsian}, its end invariants $\nu^+(\rho)$ and $\nu^-(\rho)$
are a pair of hyperbolic structures in the Teichm\"uller space
$\mathcal{T}(S)$.  In the general setting, each end invariant
$\nu^\pm(\rho)$ is a disjoint union of a multicurve on 
$S$, the {\em  parabolic locus}, with either an {\em ending lamination} or a
complete finite-area hyperbolic structure supported on each
complementary component. A curve $c$ lies in the parabolic  locus of $\nu^+(c)$
if it is an upward-pointing parabolic curve, i.e. $\rho(c)$ is parabolic and, after one
chooses an orientation-preserving identification of $N_\rho=\mathbb H^3/\rho(\pi_1(S))$ with $S\times\R$ in the homotopy
class determined by $\rho$,  the
cusp of $N_\rho$ associated to $c$ lies in $S\times [r,\infty)$ for some $r\in\R$. Similarly, a curve lies
in the parabolic locus of $\nu^-(\rho)$ if and only if it is a downward-pointing parabolic curve.

Given an end invariant $\nu$ for $\rho$ and a curve $d$ in ${\mathcal{C}}(S)$, the
{\em curve complex} of $S$, 
we define the {\em length} $l_\nu(d)$ to be 0 if $d$ is a curve
in $\nu$, to be hyperbolic length $l_\tau(d)$ if $d$ lies in a
subsurface $R$ admitting a complete hyperbolic structure $\tau$
induced by $\rho$, and to be $\infty$ otherwise.  
A collection of non-homotopic essential simple closed curves $\mu$ on $S$ is 
{\em  binding} if  any representative of $\mu$ on $S$ decomposes $S$ into
disks or peripheral annuli.
We call a fixed choice of such a
collection $\mu$ a {\em coarse basepoint} for $\mathcal{C}(S)$.  We
define
$$m(\nu,d,\mu)=\max\left\{\sup_{d\subset\partial Y}d_Y(\nu,\mu),{1\over l_{\nu}(d)}\right\}$$
where the supremum in the first term is taken over all essential
subsurfaces $Y$ with $d$ contained in $\partial Y$, and the {\em
  subsurface projection} $d_Y(\nu,\mu)$ is a measure of the distance
in $\CC(Y)$ between projections $\pi_Y(\nu)$ and $\pi_Y(\mu)$ to
$\CC(Y)$ of $\nu$ and $\mu$ (see sections \ref{curve complexes} and
\ref{end invariants}).

If we take the supremum of $d_Y(\nu,\mu)$ only over non-annular
surface with boundary containing $d$ (i.e. $Y$ is not isotopic to a collar neighborhood $\collar(d)$ of $d$), then we obtain
$$m^{na}(\nu,d,\mu)=\max \left\{ \sup_{\stackrel{d\subset\partial Y,}{\ Y\ne \collar}(d)}d_Y(\nu,\mu),
  \frac{1}{l_{\nu}(d)} \right\}.$$

Choose a coarse basepoint $\mu$ in $\mathcal{C}(S)$ once and for all.
We say that a sequence $\{\nu_n^\pm\}$ of end invariants {\em bounds
  projections} if for some $K>0$ the following conditions hold:
\begin{enumerate}
\item[(a)] 
Every geodesic in $\mathcal{C}(S)$ joining $\pi_S(\nu_n^+)$ to
   $\pi_S(\nu_n^-)$ lies at distance at most $K$ from $\mu$. 
\item[(b)] If
  $d\in\mathcal{C}(S)$ is a curve, then either 
\subitem (i) there
  exists $\beta(d) \in \{ +, -\}$ such that $\{m(\nu_n^\beta,d,\mu)\}$
  is {\em eventually   bounded}, meaning there is $ N\in\IN$ such that 
$$\sup\{  m(\nu_n^\beta,d,\mu), n\geq N\}<\infty,$$ or 
\subitem (ii)
  $\{m^{na}(\nu_n^+,d,\mu)\}$ and $\{m^{na}(\nu_n^-,d,\mu)\}$ are both
  eventually bounded and there exists $w(d)\in \Z$ and a sequence
  $\{s_n\}\subset\Z$ such that $\lim |s_n|=\infty$ and both
 $$\{d_Y(D_Y^{s_nw(d)}(\nu_n^+),\mu)\} \ \ \ \text{and} \ \ \ 
  \{d_Y(D_Y^{s_n(w(d)-1)}(\nu_n^-),\mu)\}$$ are eventually bounded
  when $Y=\collar(d)$ and $D_Y$ is the right Dehn-twist about $Y$.
\end{enumerate}

In this definition, we say that a curve $d$ is a {\em combinatorial parabolic} if $\{m(\nu_n^+,d,\mu)\}$ or $\{m(\nu_n^-,d,\mu)\}$ is not eventually bounded. It is an {\em  upward-pointing combinatorial parabolic} if $\{m(\nu_n^+,d,\mu)\}$ is not eventually bounded  and $\{m(\nu_n^-,d,\mu)\}$ is
eventually bounded. Similarly, we say that a curve $d$ is a {\em  downward-pointing combinatorial
parabolic} if $\{m(\nu_n^-,d,\mu)\}$ is not eventually bounded  and $\{m(\nu_n^+,d,\mu)\}$
is eventually bounded. 
We say that $d$ is a {\em combinatorial wrapped parabolic} if both $\{m(\nu_n^+,d,\mu)\}$ and
$\{m(\nu_n^-,d,\mu)\}$ are unbounded. If $d$ is combinatorial parabolic, then we
we say that $w(d)$ is its {\em combinatorial wrapping number}. We notice that all these definitions
are independent of the choice of coarse basepoint, so we will usually choose our coarse basepoint
to be a complete marking of $S$ (see Section \ref{curve complexes}).

We will see that, for a convergent sequence, every
combinatorial parabolic is indeed associated to a parabolic in the
limit and furthermore that one can determine which
side the parabolic manifests on directly from the asymptotic behavior
of $\{m(\nu_n^+,d,\mu)\}$ and $\{m(\nu_n^-,d,\mu)\}$. Moreover, every
wrapped parabolic is associated to the wrapping of an immersion of a
compact core for $N_\rho$ in a geometric limit of $\{ N_{\rho_n}\}$.

We combine our results with \cite[Theorem 1.3]{BBCM} to see 
that the asymptotic behavior of the end invariants predicts the curve and
lamination components of the end invariants of the limit. 

We also describe, in the case when $N_{\rho_n}$ converges
geometrically to a hyperbolic 3-manifold, how a compact core for the
algebraic limit is ``wrapped'' when pushed down into the geometric
limit. We describe this phenomenon in terms of a wrapping multicurve
and an associated wrapping number (we refer the reader to section
\ref{define wrapping numbers} for definitions).  Anderson and Canary
\cite{AC-pages} first observed that there need not be a compact core
for the algebraic limit that embeds in the geometric limit and
McMullen \cite[Lemma A.4]{mcmullen} gave the first description of this
phenomenon in the surface group case. We show that there is a compact
core for the algebraic limit that embeds in the geometric limit if
and only if the wrapping multicurve is empty.

\begin{theorem}{predictive power}
Suppose that $\{\rho_n\}$ is a sequence in $AH(S)$ converging to
$\rho\in AH(S)$ and $\{\nu_n^{\pm}\}$ bounds projections. Then
\begin{enumerate}
\item
$\ell_\rho(d)=0$ if and only if $d$ is a combinatorial parabolic for the sequence
$\{\nu_n^\pm\}$,
\item
A parabolic curve $d$ is upward-pointing in $N_\rho$ if and only if
$$|m(\nu_n^+,d,\mu)|-|m(\nu_n^-,d,\mu)|\to +\infty.$$
\item A lamination $\lambda\in {\mathcal{EL}}(Y)$ is an ending
  lamination for an upward-pointing (respectively downward-pointing)
  geometrically infinite end for $N_\rho$ if and only if
  $\{\pi_Y(\nu_n^+)\}$ (respectively $\{\pi_Y(\nu_n^-)\}$) converges
  in  ${\mathcal{C}}(Y)\cup{\mathcal{EL}}(Y)$ to $\lambda$.
\item If $\{\rho_n(\pi_1(S))\}$ converges geometrically to
  $\hat\Gamma$, then the wrapping multicurve for
  $(\{\rho_n\},\rho,\hat\Gamma)$ is the collection of combinatorial wrapping parabolics
  given by $\{\nu_n^\pm\}$ and if $d$ is a wrapping parabolic,
  then the combinatorial wrapping number $w(d)$ agree with the actual
  wrapping number $w^+(d)$.
\item There is a compact core for $N_\rho$ that embeds in $\hat
  N={\bf H}^3/\Gamma$ if and only if there are no combinatorial wrapping
  parabolics.
\end{enumerate}
\end{theorem}

\medskip

We also obtain the following alternative characterization of
convergence in terms of sequence of bounded length multicurves in
$N_{\rho_n}$.

\begin{theorem}{multi-curve version}
  Let $S$ be a compact, orientable surface and let $\{\rho_n\}$ be a
  sequence in $AH(S)$. Then $\{\rho_n\}$ has a convergent subsequence
  if and only if there exists a subsequence $\{\rho_j\}$ of
  $\{\rho_n\}$ and a sequence $\{c_j^\pm\}$ of pairs of multicurves
  so that $\{\ell_{\rho_j}(c_j^+\cup c_j^-)\}$ is bounded and
  $\{c_j^\pm\}$ bounds projections.
\end{theorem}

When $c$ is a multicurve and $d$ is a curve, we define
$$m(c,d,\mu)=\sup_{d\subset\partial Y}d_Y(c,\mu)$$
if $i(c,d)\ne 0$ and $m(c,d,\mu)=\infty$ otherwise.  Similarly, we
define $$m^{na}(c,d,\mu)= \sup_{\stackrel{d\subset\partial Y,}{\ Y\ne
    \collar(d)}}d_Y(c,\mu).$$

In analogy with the end invariants situation, 
we say that a sequence $\{ c_n^\pm\}$ of pairs of multicurves {\em
  bounds projections} if, 
choosing a coarse basepoint $\mu$ in $\mathcal{C}(S)$,
the following conditions hold:
\begin{enumerate}
\item[(a)] every geodesic joining $\pi_S(c_n^+)$ to $\pi_S(c_n^-)$
lies a bounded distance from $\mu$ in $\mathcal{C}(S)$,
\item[(b)] if $d\in\mathcal{C}(S)$ is a curve, then either \subitem (i) there
  exists $\beta(d)$ such that $\{m(c_n^\beta,d,\mu)\}$ is eventually
  bounded, or \subitem (ii) $\{m^{na}(c_n^+,d,\mu)\}$ and
  $\{m^{na}(c_n^-,d,\mu)\}$ are both eventually bounded and there
  exists $w(d)\in \Z$ and a sequence $\{s_n\}\subset\Z$ such that
  $\lim |s_n|=\infty$ and both 
$$\{d_Y(D_Y^{s_nw(d)}(c_n^+),\mu)\} \ \ \ \text{and} \ \ \ 
\{d_Y(D_Y^{s_n(w(d)-1)}(c_n^-),\mu)\}$$ are eventually bounded when
  $Y=\collar(d)$ and $D_Y$ is the right Dehn-twist about $Y$.
\end{enumerate}

\medskip 

We again say that a curve $d$ is an {\em  upward-pointing combinatorial
parabolic} if $\{m(c_n^+,d,\mu)\}$ is not eventually bounded  and $\{m(c_n^-,d,\mu)\}$ is
eventually bounded. Similarly, we say that a curve $d$ is a {\em  downward-pointing combinatorial
parabolic} if $\{m(c_n^-,d,\mu)\}$ is not eventually bounded  and $\{m(c_n^+,d,\mu)\}$
is eventually bounded. We say that $d$ is a {\em combinatorial wrapped parabolic}
if both $\{m(c_n^+,d,\mu)\}$ and $\{m(c_n^-,d,\mu)\}$ are
unbounded. However, unlike in the end invariant case, the bounded
length multicurves bounding projections need not predict the ending
laminations or the parabolics in the algebraic limit.  For example, if
$\{\rho_n\}$ is a convergent sequence, then any constant sequence
$\{c_n^\pm\}=\{c^\pm\}$ of pairs of filling multicurves will bound projections.  
We will discuss this issue further in section \ref{suff}.

\bigskip
\noindent{\bf Hausdorff limits of end invariants.}
We note that Theorems 3--6  and 12 of Ohshika \cite{ohshika-divergence}, which
discuss matters of convergence  and divergence of Kleinian groups in the context of
convergence of end invariants in the measure and Hausdorff topology on
laminations, are special cases of Theorems \ref{main theorem} and
\ref{predictive power}.  The failure of any
of these more traditional forms of convergence of laminations to predict
completely the end invariant of the limit, and in turn the presence of a
convergent subsequence, is an essential point of the present
discussion. 
The following examples motivate the need for the 
use of {\em  subsurface projections} to capture convergence phenomena, both here and in \cite{BBCM}.

\begin{example}{non convergence example}
We use a variation of a construction of Brock \cite[Theorem 7.1]{brock-invariants} to produce sequences $\{\rho_n^1\}$ and $\{\rho_n^2\}$ 
in $AH(S)$, so that the ending invariants of $\{\rho_n^1\}$ and $\{\rho_n^2\}$ have the same Hausdorff limit
and  $\{\rho_n^1\}$ and $\{\rho_n^2\}$ have convergent subsequences with algebraic limits whose parabolic loci 
differ. We further construct sequences
$\{\rho^3_n\}$  and $\{\rho_n^4\}$  in $AH(S)$ so that the ending invariants of $\{\rho_n^3\}$ and $\{\rho_n^4\}$
have the same Hausdorff limit, and $\{\rho_n^3\}$ has a  convergent subsequence, but
$\{\rho_n^4\}$  does not have a convergent subsequence.

We first choose a non-separating curve $\alpha$ on $S$ and a mapping class $\psi$ which restricts to 
a pseudo-Anosov diffeomorphism of $S-\collar(\alpha)$. We then choose a non peripheral curve $\gamma$ in
\hbox{$S-\collar(\alpha)$} and a pants decomposition $c_0^1$ of $S$, such that all curves in $c_0^1$ cross $\alpha$.
Let \hbox{$c_n^1= D_\gamma^n \circ \psi^n (c_0^1)$} where $D_\gamma$ is a
Dehn-twist about $\gamma$.  Adjusting if necessary by  Dehn twists
$D_\alpha^{k_n}$ for suitable powers $k_n$, the multicurves $\{c_n^1\}$
converge to a Hausdorff limit $\lambda_H$ which contains $\gamma$ and intersects $\alpha$ transversely.
The lamination $\lambda_H$ spirals about $\gamma$  and gives a decomposition of $S\setminus\gamma$ into ideal polygons.
One can check that $\{m(c_n^1,d,\mu)\}$ is bounded if $d$ is not either $\alpha$ or  $\gamma$, and
that $m^{na}(c_n^1,\alpha,\mu)\to\infty$ and  $m(c_n^1,\gamma,\mu)\to\infty$.

Since $\lambda_H$ is a limit of multicurves and gives a decomposition of \hbox{$S\setminus\alpha$} into ideal polygons, 
one can find a pants decomposition $c_0^2$ of $S$ such that \hbox{$\{c_n^2=D_\gamma^n(c_0^2)\}$}
converges to  $\lambda_H$.
One can check that $\{m(c_n^2,d,\mu)\}$ is bounded if $d$ is not $\gamma$ and that $m(c_n^2,\gamma,\mu)\to\infty$.

Let $a$ be a pants decomposition of $S$ which crosses both $\alpha$ and $\gamma$. Let $\rho_n^1$ have
top ending invariant $c_n^1$ and bottom end invariant $a$, while $\rho_n^2$ has top end invariant $c_n^2$
and bottom end invariant $a$. The Hausdorff limit of the top ending invariants of both $\{\rho_n^1\}$ and $\{\rho_n^2\}$
is $\lambda_H$, while the Hausdorff limit of the bottom ending invariants of each sequence is $a$. Theorem
\ref{main theorem} implies that both $\{\rho_n^1\}$ and $\{\rho_n^2\}$ have convergent subsequences.
Theorem \ref{predictive power} implies that if $\rho^1_\infty$ is the algebraic limit of any convergent
subsequence of $\{\rho_n^1\}$, then the upward-pointing parabolic locus of $\rho^1_\infty$ is $\alpha\cup\gamma$,
while the downward-pointing parabolic locus is $a$. On the other hand, if $\rho^2_\infty$ is the algebraic limit of 
any convergent subsequence of $\{\rho_n^2\}$, then the upward-pointing parabolic locus of $\rho^2_\infty$ is 
$\gamma$, while the downward-pointing parabolic locus is $a$.

Let $b$ be a pants decomposition of $S$ which crosses $\gamma$ and contains $\alpha$. Let $\rho_n^3$ have
top ending invariant $c_n^2$ and bottom end invariant $b$, while $\rho_n^4$ has top ending invariant $c_n^1$
and bottom end invariant $b$. The Hausdorff limit of the top ending invariants of both $\{\rho_n^3\}$ and $\{\rho_n^4\}$
is $\lambda_H$, while the Hausdorff limit of the bottom end invariants of each sequence is $b$. Theorem
\ref{main theorem} implies that $\{\rho_n^3\}$ has a convergent subsequence, but that $\{\rho_n^4\}$ does not
have a convergent subsequence.
\end{example}

\begin{example}{}
If one regards the Hausdorff limit of the end invariants of a sequence of quasifuchsian groups as the Hausdorff 
limit of a sequence of minimal length pants decompositions in the associated conformal structures, 
as Ohshika \cite{ohshika-divergence} does, then one
may use the wrapping construction to construct simpler examples.

Let $\alpha$ be a non-peripheral curve on $S$. Let $X$ be a hyperbolic surface with unique minimal length
pants decomposition $r$ which crosses $\alpha$.
Let $\tau_n^1$ be a quasifuchsian group with top end invariant $D_\alpha^{3n}(X)$
and bottom end invariant $D_\alpha^{2n}(X)$. The Hausdorff limit of the top and bottom end invariants of $\{\tau_n^1\}$
is the lamination $\lambda$ obtained by ``spinning'' $r$ about $\alpha$. Theorem \ref{main theorem}
implies that $\{\tau_n^1\}$ has a convergent subsequence, while Theorem \ref{predictive power} implies that
if $\tau^1_\infty$ is the algebraic limit of any convergent subsequence of $\{\tau_n^1\}$, then the upward-pointing
parabolic locus of $\tau^1_\infty$ is $\alpha$, while the downward pointing parabolic locus is empty.

Let $\tau_n^2$ be a quasifuchsian group with top end invariant $D_\alpha^{n}(X)$
and bottom end invariant $D_\alpha^{2n}(X)$. The Hausdorff limit of the top and bottom end invariants of $\{\tau_n^2\}$
is again $\lambda$. Theorem \ref{main theorem}
implies that $\{\tau_n^2\}$ has a convergent subsequence, while Theorem \ref{predictive power} implies that
if $\tau^2_\infty$ is the algebraic limit of any convergent subsequence of $\{\tau_n^2\}$, then the upward-pointing
parabolic locus of $\tau^2_\infty$ is empty, while the downward pointing parabolic locus is $\alpha$.

Let $\tau_n^3$ be a quasifuchsian group with top end invariant $D_\alpha^{2n}(X)$
and bottom end invariant $D_\alpha^{2n}(X)$. The Hausdorff limit of the top and bottom end invariants of $\{\tau_n^3\}$
is again $\lambda$. Theorem \ref{main theorem}
implies that $\{\tau_n^3\}$ has no convergent subsequences.
\end{example}

\bigskip\noindent{\bf Outline of the paper:} In section
\ref{background} we recall definitions and previous results that will
be used in the paper.  In section \ref{wrap1} we define the wrapping
multicurve and the wrapping numbers. We assume that $\{\rho_n\}$
converges to $\rho$ and that $\{N_{\rho_n}\}$ converges geometrically
to $\hat N$. Let $\pi:N_\rho\to \hat N$ be the obvious covering map.
We first find a level surface $F$ in $N_\rho$ and a collection $Q$ of
incompressible annuli in $F$, so that $\pi|_F$ is an immersion,
$\pi|_{F-Q}$ is an embedding and $\pi$ wraps $Q$ around the boundary
of a cusp region in $\hat N$. The collection $q$ of core curves of
elements of $Q$ is the wrapping multicurve. The wrapping number then
records ``how many times'' $Q$ is wrapped around the cusp region.

In section \ref{nec}, we prove that if a sequence $\{\rho_n\}\subset
AH(S)$ converges, then some subsequence of its end invariants predicts
convergence. We also establish Theorem \ref{predictive power}.  We
first use work of Minsky \cite{minsky:kgcc,ELC1} and
Brock-Bromberg-Canary-Minsky \cite{BBCM} to establish the results in
the case that the wrapping multicurve is empty. When the wrapping
multicurve is non-empty, we use the wrapped surface $F$ from section
\ref{wrapping} to construct two new sequences, that differ from the
original sequence by powers of Dehn twists in components of the
wrapping multicurve, but themselves have empty wrapping
multicurves. We can then apply the results from the empty wrapping
multicurve case to both of these sequences. Analyzing the
relationship between the end invariants of the original sequence and
the two new sequences allow us to complete the proof.

In section \ref{equiv}, we show that if the sequence $\{\nu_n^\pm\}$
of end invariants for a sequence $\{\rho_n\}$ in $AH(S)$ bounds projections,
then one can find a subsequence $\{\rho_j\}$ and a sequence
$\{c_j^\pm\}$ of pairs of multicurves such that
$\{\ell_{\rho_j}(c_j^+\cup c_j^-)\}$ is bounded and $\{c_j^\pm\}$
bounds projections. The difficulty comes from the fact that one must insure that $c_n^+$ and $c_n^-$ do not
share any curves while bounding projections. In particular one must take special care of the curves
where $\{m(\nu_n^\beta,d,\mu)\}$ is unbounded.
To overcome these difficulties, we will construct $ c_n^\pm$ as minimal length pants decompositions under some constraints.

In section \ref{suff}, we show that if $\{\rho_n\}$ is a sequence in
$AH(S)$ and there is a sequence of bounded length multicurves 
$\{c_n^\pm\}$ that bound projections, then $\{\rho_n\}$ has a
convergent subsequence. Again we start with the case that the wrapping
multicurve is empty.  We may assume that each $c_n^\pm$ is a pants
decomposition of $S$.  We first use results of Minsky
\cite{minsky:kgcc} to find a pants decomposition $r$ such that
$\{\ell_{\rho_n}(r)\}$ is bounded.  We then construct the model
manifold $M_n^\beta$ associated to the hierarchy joining $r$ to
$c_n^\beta$ and observe, using work of Bowditch \cite{bowditch} and
Minsky \cite{ELC1}, that there is a uniformly Lipschitz map of
$M_n^\beta$ into $N_{\rho_n}$.  (If $r$ and $c_n^\beta$ share curves
we consider a model manifold associated to a subsurface of $S$.)  We
find a bounded length transversal in $M_n$ to each curve in $r$ and
then observe that it also has bounded length in $N_{\rho_n}$. We pass
to a subsequence so that the sequence of transversals we have
constructed is constant and then simply apply the Double Limit Theorem
to conclude that there is a convergent subsequence. When the wrapping
multicurve is not empty, we construct two new sequences with empty
wrapping multicurves and use them to produce a converging subsequence
of the original sequence.

Finally, in section \ref{conc} we combine the results of sections
\ref{nec}, \ref{equiv} and \ref{suff} to complete the proofs of both
Theorems \ref{main theorem} and \ref{multi-curve version}.

\section{Background}
\label{background}

In this section, we collect definitions and previous results which will be used in the paper.
We first need to recall the definitions of curve complexes of subsurfaces, subsurface projections, markings and end invariants.

\subsection{Curve complexes, markings and subsurface projections}	\label{curve complexes}

If $W$ is an essential non-annular subsurface of $S$, its curve complex ${\mathcal{C}}(W)$ 
is a locally infinite simplicial complex whose vertices are isotopy classes of essential non-peripheral curves
on $W$. Two vertices are joined by an edge if and only if the associated curves intersect minimally. A collection
of $n+1$ vertices span a $n$-simplex if the corresponding curves have mutally disjoint representatives. Masur
and Minsky \cite{masur-minsky} proved that ${\mathcal{C}}(W)$ is Gromov hyperbolic with respect to its
natural path metric.

We will assume throughout that all curves are essential and non-peripheral. A {\em multicurve}
will be a collection of disjoint curves, no two of which are homotopic. A {\em pants decomposition} of $W$
is a maximal multicurve.

Klarreich \cite{klarreich}, see also Hamenstadt \cite{hamenstadt}, showed that the 
Gromov boundary $\partial_\infty{\mathcal{C}}(W)$
of ${\mathcal{C}}(W)$ can be naturally identified with the space ${\mathcal{EL}}(W)$ of filling geodesic
laminations on $W$.

A {\em marking} $\mu$ on $S$ is a multicurve ${\rm base}(\mu)$
together with a selection of transversal curves, at most one for each
component of ${\rm base}(\mu)$. A transversal curve to a curve $c$ in
${\rm base}(\mu)$ intersects $c$ and is disjoint from ${\rm
  base}(\mu)-c$. A marking is {\em complete} if ${\rm base}(\mu)$ is a
pants decomposition and every curve in ${\rm base}(\mu)$ has a
transversal. A {\em generalized marking} is a collection of filling
laminations on a disjoint collection of subsurfaces together with the
boundary of those subsurfaces and a marking of their complement.  (See
Masur-Minsky \cite{masur-minsky2} and Minsky \cite{ELC1} for a more
careful discussion of markings and generalized markings.)

If $W$ is an essential non-annular subsurface, one may define a {\em subsurface projection}
$$\pi_W:{\mathcal{C}}(S)\to {\mathcal{C}}(W)\cup \{\emptyset\}.$$
If $c\in{\mathcal{C}}(S)$ and $c$ is disjoint from
$W$, then $\pi_W(c)=\emptyset$. If not, $c\cap W$ is a collection of arcs and curves on $W$. Each arc in
$c\cap W$ may be surgered to produce an essential curve on $W$ by
adding arcs in $\partial W$.  
We let $\pi_W(c)$ denote a choice of one of the resulting
essential curves in $W$; then $\pi_W(c)$ is {\em coarsely}
well-defined - any two choices lie at bounded distance (see \cite[Lemma 2.3]{masur-minsky2}).
For a subset $\mu$ of ${\mathcal C}(S)$ (such as a multicurve, a marking or a coarse basepoint for ${\mathcal{C}(S)}$), 
we choose $\pi_W(\mu)$ to be a curve in $\bigcup_{c\in\mu}\pi_W(c)$ if there is one and to be $\emptyset$ otherwise. 
We can then define
$$d_W(c,\mu)=d_{{\mathcal{C}}(W)}(\pi_W(c),\pi_W(\mu))$$
if $\pi_W(c)\neq\emptyset$ and $\pi_W(\mu)\neq\emptyset$, and define
$d_W(c,\mu)=+\infty$ otherwise.
 
If $\mu$ is a generalized marking on $S$, 
then we define 
$$\pi_W(\mu)\in{\mathcal{C}}(W)\cup{\mathcal{EL}}(W)\cup\emptyset$$ 
by 
\begin{enumerate} 
\item letting $\pi_W(\mu) = \emptyset$ if $\mu$ does not intersect $W$, 
\item letting $\pi_W(\mu) = \lambda$ if $\lambda \subset \mu$ lies in
$\mathcal{EL}(W)$, 
\item constructing $\pi_W(\mu)$ as above using any simple closed curve
  or proper arc in $\mu \cap W$.
\end{enumerate}

For a pair of generalized markings, we define
$$d_W(\mu,\mu')=d_{{\mathcal{C}}(W)}(\pi_W(\mu),\pi_W(\mu'))$$
if $\pi_W(\mu),\pi_W(\mu')\in{\mathcal{C}}(W)$
and $d_W(\mu',\mu)=\infty$ if $\pi_W(\mu)$ or $\pi_W(\mu')$ lies in \hbox{${\mathcal{EL}}(W)\cup\{\emptyset\}$}.

If $W$ is an essential annulus in $S$ we may also define $d_W(c,d)$ and $d_W(c,\mu)$. The simplest way to do this is
to first fix a hyperbolic metric on $S$ and let $\tilde S$ be the annular cover  $S$ so that $W$ lifts to a compact core for 
$\tilde S$. We then
compactify $\tilde S$ by its ideal boundary to obtain an annulus $A$ and define a complex ${\mathcal{C}}(W)$ whose
vertices are geodesics in $A$ that joins the two boundary components of $A$. 
We join two vertices if they have disjoint representatives. 
If we give ${\mathcal{C}}(W)$ the natural path metric then $d_{\CC(W)}(a,b) =i(a,b)+1$ and it follows that $\CC(W)$ is quasi-isometric to $\Z$. 
Given a simple closed curve $c\subset S$, we realize it as geodesic and then consider its pre-image in $\tilde S$. 
If $c$ intersects $W$ essentially, the pre-image contains an essential arc $\tilde c$ whose closure joins the two boundary components of $A$, 
we set $\pi_W(c)=\tilde c$ and we set $\pi_W(c)=\emptyset$ otherwise. 
For a subset $\mu$ of ${\mathcal C}(S)$,
we  again choose $\pi_W(\mu)$ to be an element of $\bigcup_{c\in\mu}\pi_W(c)$ if there is one and to 
be $\emptyset$ otherwise.  We can then define
$$d_W(c,\mu)=d_{{\mathcal{C}}(W)}(\pi_W(c),\pi_W(\mu))$$
if $\pi_W(c)\neq\emptyset$ and $\pi_W(\mu)\neq\emptyset$, and define $d_W(c,\mu)=+\infty$ otherwise. One can check that this definition is independent of the choice of metric.
(Again see Masur-Minsky \cite{masur-minsky2} and Minsky \cite{ELC1} for a complete discussion of subsurface projections
and the resulting distances.)

In all cases, the distance between two curves (or markings) is bounded above by a function of their
intersection number.

\begin{lemma}{intersection bounds distance}
{\rm (\cite[Lemma 2.1]{masur-minsky})}
If $S$ is a compact orientable surface, $\alpha, \beta$ are multicurves or markings on $S$ and $W$ is an essential subsurface of $S$,
then 
$$d_W(\alpha,\beta)\le 2i(\alpha,\beta)+1.$$
\end{lemma}

The following estimate is often useful in establishing relationships
between subsurface projections. Behrstock (\cite[Theorem 4.3]{behrstock}) first gave a version with inexplicit
constants which depends on the surface $S$. We will use a version, due to Leininger, with explicit 
universal constants.

\begin{lemma}{inequality on triples}
{\rm (\cite[Lemma 2.13]{johanna})}
Given a compact surface $S$, two essential subsurfaces $Y$ and $Z$
which overlap and a generalized marking $\mu$ which intersects both $Y$ and $Z$, then
$$d_Y(\mu,\partial Z)\geq 10\Longrightarrow d_Z(\mu,\partial Y)\leq 4$$
\end{lemma}

We will also use the fact that a sequence of curves which is not eventually constant blows up on some subsurface.

\begin{lemma}{blow up subsurface}
Given a sequence of simple closed curves $\{c_ n\}$ and a complete marking $\mu$ on a compact surface $S$, 
there  is a subsequence $\{c_j\}$ such that either $\{c_j\}$ is constant or there is a
subsurface $Y\subseteq S$ with  $d_Y(\mu,c_j)\longrightarrow\infty$.
\end{lemma}

\begin{proof}
Fix a metric on $S$ and realize the sequence $\{c_n\}$ as a sequence of closed geodesics. We then extract a subsequence 
$\{c_j\}$ that converges in the Hausdorff topology on closed subsets of $S$ to a geodesic lamination $\lambda$. 
If $\lambda$ contains an isolated simple closed curve, then $\{c_ j\}$ is eventually constant and we are done.
If not, let $Y$ be the supporting subsurface of a minimal sublamination $\lambda_0$ of $\lambda$.  
If $Y$ is  not an annulus, then $\lambda_0\in \mathcal{EL}(S)$ and results of Klarreich \cite[Theorem 1.4]{klarreich}
(see also Hamenstadt \cite{hamenstadt}) imply
that $d_Y(\mu,c_ j)\to\infty$. 

If $\lambda_0$ is a simple closed geodesic, then $Y =\collar(\lambda_0)$
is an annulus and, since $\lambda$ doesn't contain an isolated simple close curve, there must be leaves of $\lambda$ 
spiraling around $\lambda_0$.
Let $\tilde S_0$ be the annular cover  of $S$ associated  to the cyclic subgroup of $\pi_1(S)$ generated by $\lambda_0$
and let $\tilde \lambda_0$ be the unique lift of $\lambda_0$ to $\tilde S_0$. 
Let $\tilde c_j = \pi_Y(c_j)$.
Since $\{c_j\}$ converges to $\lambda$ in the Hausdorff topology and there exist leaves of $\lambda$ spiraling about
$\lambda_0$,
the acute angle between $\tilde c_j$ and $\tilde\lambda_0$ converges to $0$. 
It follows that $i(\tilde c_j, \tilde a) \to \infty$ for any fixed element $\tilde a\in\CC(Y)$.
In particular, if $d$ is a component of $\mu$ that intersects $\lambda_0$ and $\tilde{d}=\pi_Y(d)$, then
$$d_Y(c_j,d) = d_{\CC(Y)}(\tilde c_j, \tilde d) = i(\tilde c_j, \tilde d)+1\to\infty.$$
It follows that $d_Y(c_j,\mu)\to\infty$ as desired.
\end{proof}

\subsection{End invariants}	
\label{end invariants}

If $\rho\in AH(S)$, the end invariants of $N_\rho$ encode the asymptotic geometry of $N_\rho={\bf H}^3/\rho(\pi_1(S))$.
The Ending Lamination Theorem (see Minsky \cite{ELC1} and Brock-Canary-Minsky \cite{ELC2})  asserts 
that a representation $\rho\in AH(S)$ 
is uniquely determined by its end invariants. The reader will find a more extensive discussion of  the definition of the end invariants
and the  Ending Lamination Theorem in Minsky \cite{ELC1}.

A  $\rho(\pi_1(S))$-invariant collection $\mathcal{H}$ of disjoint horoballs in ${\bf H}^3$ is a {\em precisely invariant collection of horoballs} for
$\rho(\pi_1(S))$ if there is a horoball based at the fixed point of every parabolic element of $\rho(\pi_1(S))$ 
(and every horoball in $\mathcal{H}$ is based at a parabolic fixed point). 
The existence of such a collection is a classical consequence of the Margulis Lemma, see \cite[Proposition VI.A.11]{maskit} for example.
We define
$$N^0_\rho=({\bf H}^3-\bigcup_{H\in{\mathcal{H}}} H)/\rho(\pi_1(S)).$$
If $\mathcal{H}_p$ denotes the set of horoballs in $\mathcal{H}$ which are associated to peripheral elements
of $\pi_1(S)$, then we define
$$N^1_\rho=({\bf H}^3-\bigcup_{H\in{\mathcal{H}_p}} H)/\rho(\pi_1(S)).$$

A {\em relative compact core} for $N^0_\rho$ is a compact submanifold $M_\rho$ of $N_\rho^0$ such that the inclusion of
$M_\rho$ into $N_\rho$ is a homotopy equivalence and $M_\rho$ intersects each component of $\partial N_\rho^0$ in an
incompressible annulus. Let $P_\rho=M_\rho\cap\partial N^0_\rho$ and let
\hbox{$P^1_\rho=M_\rho\cap\partial N^1_\rho$}. (See Kulkarni-Shalen \cite{kulkarni-shalen} and McCullough
\cite{mccullough} for proofs that $N_\rho^0$ admits a relative compact core.)

Bonahon \cite{bonahon} showed that  there is an orientation preserving homeomorphism from $S\times\R$ to $N_\rho^1$
in the homotopy class determined by $\rho$. 
We will implicitly identify $N_\rho^1$ with $S\times\R$ throughout the paper.
Suppose that $W$ is a subsurface of $S$ and $f:W\to N_\rho^1$ is a map of $W$ into $S\times \R$ 
(in the homotopy class associated to $\rho|_{\pi_1(W)}$). We say that $f$  (or $f(W)$) is a {\em level subsurface}
if it is an embedding which is isotopic to  $W\times \{0\}$. If $W=S$, we say $f$ (or $f(S)$) is a {\em level surface}.

The {\em conformal boundary} $\partial_c N_\rho$ of $N_\rho$ is the quotient  by $\Gamma$ of the domain $\Omega(\rho)$ of
discontinuity for the action of $\rho(\Gamma)$ on $\hat\C$.
One may identify the conformal boundary $\partial_cN_\rho$ with a collection of components of $\partial M_\rho-P_\rho$. 
The other components of $\partial M_\rho-P_\rho$ bound neighborhoods of geometrically infinite ends of $N_\rho^0$. If $E$ is a geometrically infinite end with a neighborhood bounded by a component $W$ of $\partial M_\rho-P_\rho$, 
then there
exists a sequence  $\{\alpha_n\}\subset{\mathcal{C}}(W,\rho,L_1)$, for some $L_1=L_1(S)>0$,
whose geodesic representatives $\{\alpha_n^*\}$ exit $E$ (see Lemma \ref{bounded curves above in ends} for
a more careful statement).
The sequence $\{\alpha_n\}$ converges to an ending lamination $\lambda\in{\mathcal{EL}}(S)$ and we call
$\lambda$ the {\em ending lamination} of $E$ ($\lambda$ does not depend on the choice of the sequence $\{\alpha_n\}$). Moreover, if $\{\beta_n\}$ is any sequence in $\mathcal{C}(W)$
which converges to $\lambda$, then the sequence $\{\beta_n^*\}$ of geodesic representatives in $N_\rho$
exits $E$. (See Bonahon \cite{bonahon} for an extensive discussion of geometrically infinite ends.)

There exists an orientation-preserving homeomorphism of $S\times I$ with $M_\rho$, again in the homotopy
class determined by $\rho$, so that $\partial S\times I$ is identified with $P_\rho^1$.
Let $P^+_\rho$ denote the components of $P_\rho$ contained in $S\times \{1\}$ and let
$P^-_\rho$ denote the component of $P_\rho$ contained in $S\times \{0\}$.
A core curve of a component of $P^+_\rho$ is called an {\em upward-pointing parabolic curve} and
a core curve of a component of $P^-_\rho$ is called a {\em downward-pointing parabolic curve}.
Similarly, a component of $\partial_cN_\rho$ or a geometrically infinite end of $N_\rho$ is called {\em upward-pointing}
if it is identified with a a subsurface of $S\times \{1\}$, and is called {\em downward-pointing} if it is identified
with a subset of $S\times \{0\}$.

The {\em end invariant} $\nu_\rho^+$ consists of the multicurve $p^+$ of upward-pointing parabolic curves
together with a conformal structure on each geometrically finite component of $S\times \{1\}-p^+$, coming
from the conformal structure on the associated component of the conformal boundary,
and a filling lamination on each geometrically infinite component, which is the ending
lamination of the associated end. The end invariant $\nu_\rho^-$ is defined
similarly.

If $\nu$ is an end invariant, we define an associated generalized marking $\mu(\nu)$.
We let ${\rm base}(\mu(\nu))$ consist of all the curve and lamination
components of $\nu$ together with a minimal length pants decomposition of the conformal (hyperbolic)
structure on each  geometrically finite component.
For each curve in the minimal length pants decomposition of a geometrically finite component 
we choose a minimal length transversal. Notice that the associated marking is well-defined up
to uniformly bounded ambiguity.

Given $\rho\in AH(S)$ with end invariants $\nu^\pm$, we then define, for each essential subsurface
$W$ of $S$,
$$\pi_W(\nu^\pm)=\pi_W(\mu(\nu^\pm)).$$

Property $(3)$ in Theorem \ref{predictive power} can be viewed as a continuity property for the projections of end invariants to subsurfaces. This property was established by Brock-Bromberg-Canary-Minsky in \cite{BBCM}:

\begin{theorem}{limitend}{\rm (\cite[Theorem 1.1]{BBCM})}
Let $\rho_n\longrightarrow\rho$ in $AH(S)$. If $W\subseteq S$ is an essential subsurface of $S$, 
other than an annulus or a pair of pants, and $\lambda\in EL(W)$ is a lamination supported on $W$,
then the following statements are equivalent :
\begin{enumerate}[(1)]
\item $\lambda$ is a component of $\nu^+_\rho$.
\item $\{\pi_W(\nu^+_{\rho_n})\}$ converges to $\lambda$.
\end{enumerate}
\end{theorem}

\subsection{The bounded length curve set}

The Ending Lamination Theorem \cite{ELC1,ELC2} assures that the end
invariants coarsely determine the geometry of $N_\rho$. In particular,
one can use the end invariants to bound the lengths of curves in
$N_\rho$ and to coarsely determine the set of curves of bounded
length. We will need several manifestations of this principle.

It is often useful to, given $L>0$, consider the set of all curves in
$N_\rho$ with length at most $L$.  We define
$${\mathcal{C}}(\rho,L)=\{d\in{\mathcal{C}}(S)\ | \ \ell_\rho(d)\le L\}.$$

Minsky, in \cite{minsky:kgcc}, showed that if the projection of
${\mathcal{C}}(\rho, L)$ to ${\mathcal{C}}(W)$ has large diameter,
then $\partial W$ is short in $N_\rho$.

\begin{theorem}{kgcc fact}{\rm (\cite[Theorem 2.5]{minsky:kgcc})}
Given $S$, $\ep>0$ and $L>0$, there exists
$B(\epsilon,L)$ such that if $\rho\in AH(S)$, $W\subset S$ is a proper subsurface
and 
$${\rm diam}(\pi_W({\mathcal{C}}(\rho,L)))>B(\epsilon,L),$$
then $l_\tau(\partial W)<\epsilon$.
\end{theorem}

In \cite{BBCM} it is proven that $\pi_W({\mathcal{C}}(\rho,L))$ is
well-approximated by a geodesic joining $\pi_W(\nu^+)$ to
$\pi_W(\nu^-)$.

\begin{theorem}{bounded curves near geodesic}{\rm (\cite[Theorem
    1.2]{BBCM})}
Given $S$, there exists $L_0>0$ such that for all $L\ge L_0$, there
exists $D_0=D_0(L)$, such that, if $\rho\in AH(S)$ has end invariants
$\nu^\pm$, and $W\subset S$ is an essential subsurface more complicated
than a thrice-punctured sphere, then $\pi_W(C(\rho,L))$
has Hausdorff distance  at most $D_0$ from any geodesic in ${\mathcal{C}}(W)$ joining
$\pi_W(\nu^+)$ to $\pi_W(\nu^-)$.
Moreover, if \hbox{$d_W(\nu^+,\nu^-) > D_0$}, then 
$$C(W,\rho,L)= \{\alpha\in{\mathcal{C}}(W):l_\alpha(\rho) < L\}$$
is nonempty and also has Hausdorff distance  at most $D_0$ from any geodesic 
in ${\mathcal{C}}(W)$ joining $\pi_W(\nu^+)$ to $\pi_W(\nu^-)$.
\end{theorem}

As a generalization of Minsky's a priori bounds (see \cite[Lemma 7.9]{ELC1}),
Bowditch proved that all curves on a tight geodesic in
${\mathcal{C}}(W)$ joining two bounded length multicurves, also have
bounded length. We recall that if $W$ is a non-annular essential
subsurface of $S$, then a {\em tight geodesic} is a sequence $\{w_i\}$
of simplices in ${\mathcal{C}}(W)$ such that if $v_i$ is a vertex of
$w_i$ and $v_j$ is a vertex of $w_j$, then $d_W(v_i,v_j)=|i-j|$ and
each $w_i$ is the boundary of the subsurface filled by $w_{i-1}\cup
w_{i+1}$.

\begin{theorem}{bowditch a priori}
{\rm (Bowditch \cite[Theorem 1.3]{bowditch})}
Let $S$ be a compact orientable surface.
Given $L>0$ there exists $R(L,S)$ such that if 
\hbox{$\rho\in AH(S)$}, $W$ is an essential non-annular subsurface of $S$,
$\{w_i\}_{i=0}^n$ is a tight geodesic in ${\mathcal{C}}(W)$, and $\ell_\rho(w_0)\le L$ and
$\ell_\rho(w_n)\le L$, then
$$\ell_\rho(w_i)\le R$$
for all $i=1,\ldots,n-1$.
\end{theorem}

\subsection{Margulis regions and topological ordering}

There exists a constant $\epsilon_3>0$, known as the {\em Margulis
  constant}, such that if $\epsilon\in(0,\epsilon_3)$ and $N$ is a
hyperbolic 3-manifold, then each component of the {\em thin part}
$$N_{thin(\epsilon)}=\{x\in N\ |\ {\rm inj}_N(x)<\epsilon\}$$
is either a solid torus neighborhood of a closed geodesic or the
quotient of a horoball by a group of parabolic elements (see \cite[Corollary 5.10.2]{thurston-notes} for example). If $\rho\in
AH(S)$ and $d$ is a curve on $S$, then let $\MT_\epsilon(d)$ be the
component of $N_{thin(\epsilon)}$ whose fundamental group is generated
by $d$. With this definition, $\MT_{\epsilon}(d)$ will often be
empty. When it is non-empty, we will call it a {\em Margulis region}
and when it is non-compact we will call it a {\em Margulis cusp
  region}. Notice that if $N={\bf H}^3/\Gamma$, then the pre-image in
${\bf H}^3$ of all the non-compact components of $N_{thin(\epsilon)}$,
for any $\epsilon\in (0,\epsilon_3)$, is a precisely invariant system
of horoballs for $\Gamma$.

Suppose that $\alpha$ and $\beta$ are homotopically non-trivial curves
in $N^1_\rho$ and that their projections to $S$ intersect
essentially. We say that $\alpha$ lies above $\beta$ if $\alpha$ may
be homotoped to $+\infty$ in the complement of $\alpha$ (i.e. $\alpha$
may be homotoped into $S\times [R,\infty)$ in the complement of
$\beta$ for all $R$).  Similarly, we say that $\beta$ is {\em below}
$\alpha$ if $\beta$ may be homotoped to $-\infty$ in the complement of
$\alpha$ (see \cite[\textsection 2.5]{BBCM} for a more detailed discussion).

It is shown in \cite{BBCM} that if the geodesic representative of a curve $d$ lies above
the geodesic representative of the boundary component of  a subsurface $W$, then
the projection of $d$ lies near the projection of $\nu^+$.

\begin{theorem}{top is top}{\rm (\cite[Theorem 1.3]{BBCM} )}
Given $S$ and $L>0$ there exists $D=D(S,L)$ such that if $\alpha\in{\mathcal{C}}(S)$,
$\rho\in AH(S)$ has end invariants $\nu^\pm$, $l_\rho(\alpha)<L$, $\alpha$ overlaps  a proper subsurface $W\subset S$
(other than a thrice-punctured sphere),
and there exists a component $\beta$ of $\partial W$
such that $\alpha^*$ lies above $\beta^*$ in $N_\rho$, then
$$d_W(\alpha,\nu^+)<D.$$ 
\end{theorem}

\noindent {\bf Remark:} If $\rho(\alpha)$ is parabolic, then
$\alpha$ has no geodesic representative in $N_\rho$. If $\alpha$ is an upward-pointing parabolic,
it is natural to say that it lies above the geodesic representative of every curve it overlaps, while
if $\alpha$ is a downward-pointing parabolic,
it is natural to say that it lies below the geodesic representative of every curve it overlaps.

\bigskip

The following observation is a consequence of the geometric description of geometrically infinite ends 
(see Bonahon \cite{bonahon}).

\begin{lemma}{bounded curves above in ends}
Given a compact surface $S$, there exists $L_1=L_1(S)$ such that if
$\rho\in AH(S)$, $W$ is an essential sub-surface of $S$ which is the support of
a geometrically infinite end $E$  of $N_\rho^0$ and $\Delta$ is a finite subset of $\mathcal{C}(W)$,
then there exists a pants decomposition $r$ of $W$ such that $l_\rho(r)\leq L_1$ and any curve in $r$ lies above,
respectively below, any curve in $ \Delta $ when $E $ is upward-pointing, respectively downward-pointing.
\end{lemma}

\subsection{Lipschitz surfaces and bounded length curves}
\label{lipschitzsurf}

If $\rho\in AH(S)$, then
a {\em \hbox{$K$-Lipschitz} surface} in $N_\rho$ is a $\pi_1$-injective  \hbox{$K$-Lipschitz} map \hbox{$f:X\to N_\rho$} where $X$
is a (complete) finite area hyperbolic surface. Incompressible pleated surfaces, see Thurston \cite[section 8.8]{thurston-notes}
and Canary-Epstein-Green \cite[Chapter I.5]{CEG}, are examples of \hbox{$1$-Lipschitz} surfaces.
If $W$ is an essential subsurface of $S$ and \hbox{$\alpha\in\mathcal{C}(W)$}, then we say that a \hbox{$K$-Lipschitz}
surface $f:X\to N_\rho$, where $X$ is a hyperbolic structure on ${\rm int}(W)$, {\em realizes}  the pair $(\alpha,W)$
if there exists a homeomorphism $h:{\rm int}(W)\to X$ such that $(f\circ h)_*$ is conjugate to $\rho|_{\pi_1(W)}$
and $f(h(\alpha))=\alpha^*$.
Thurston observed that if $\rho(\pi_1(\partial W))$ is purely parabolic and $\rho(\alpha)$ is hyperbolic, then
one may always find a pleated surface realizing $(\alpha,W)$.

\begin{lemma}{realizations}
{\rm (Thurston \cite[Section 8.10]{thurston-notes}, Canary-Epstein-Green \cite[Theorem I.5.3.6]{CEG})}
Suppose that $\rho\in AH(S)$, $W$ is an essential subsurface of $S$ and $\alpha\in \mathcal{C}(W)$.
If every (non-trivial) element of $\rho (\pi_1(\partial W))$ is parabolic and $\rho(\alpha)$ is hyperbolic,
then there exists a 1-Lipschitz surface  realizing $(\alpha,W)$.
\end{lemma}

One may use Lemma \ref{realizations} and a result of Bers (\cite{bers-constant}, see also \cite[p.123]{buser}) 
to construct bounded length pants decompositions which include any fixed bounded length curve.

\begin{lemma}{bounded length pants}
Suppose that $\rho\in AH(S)$, $W$ is an essential subsurface of $S$ and $\alpha\in \mathcal{C}(W)$.
Given $L>0$, there is  $L'=L'(L,S)$  such that, if $\ell_{\rho}(\alpha)+\ell_{\rho}(\partial W)\leq L$, then
$W$ admits a pants decomposition $p$ containing $\alpha$ such that $\ell_{\rho}(p)\leq L'$.
\end{lemma}

\subsection{Geometric limits}

A sequence $\{\Gamma_n \}$ of Kleinian groups {\em converges geometrically} to a Kleinian group $\hat\Gamma$
if every accumulation point $\gamma$ of every sequence \hbox{$\{ \gamma_n\in \Gamma_n\}$} lies in $\hat\Gamma$ 
and if every element $\alpha$ of $\Gamma_\infty$ is the limit of a sequence \hbox{$\{ \alpha_n\in \Gamma_n\}$}. 
It is useful, to think of geometric convergence of a sequence of torsion-free Kleinian groups, in terms 
of geometric convergence of the sequence of hyperbolic 3-manifolds. The following result combines
standard results about geometric convergence which will be used in the paper.

\begin{lemma}{comparisons respect tubes}
Suppose that $\{\rho_n:\pi_1(S)\to {\rm PSL}(2,\mathbb C)\}$ is a sequence of
discrete faithful representations converging to the discrete faithful representation
$\rho:\pi_1(S)\to {\rm PSL}(2,\mathbb C)$.
Then, there exists a subsequence $\{\rho_j\}$ so that $\{\rho_j(\pi_1(S))\}$
converges geometrically to $\hat \Gamma$. 

Let $\hat N=\mathbb H^3/\hat\Gamma$ and let $\pi:N_\rho\to\hat N$ be the natural
covering map. Let
$\hat{\mathcal{H}}$ be a precisely
invariant system of horoballs for $\hat\Gamma$.

There exists a  nested sequence $\{Z_j\}$ of compact sub-manifolds
exhausting $\hat N$ and $K_j$-bilipschitz smooth embeddings $\psi_j:Z_j\to N_{\rho_j}$ such that:
\begin{enumerate}
\item
$K_j\to 1$.
\item
If $V$ is a  compact component of $\partial\hat N^0$, then, for all large enough $n$,
$\psi_j(\partial V)$ is the boundary of a Margulis region for $N_{\rho_j}$.
\item 
If $Q$ is a compact subset of a non-compact component of $\partial\hat N^0$, then, for all
large enough $j$, $\psi_j(Q)$ is contained in the boundary of a Margulis region $V_j$
for $N_{\rho_j}$ and $\psi_j(Z_j\cap \hat N^0)$ does not intersect $V_j$.
\item
If $X$ is a finite complex and $h:X\to N_\rho$ is continuous, then, for all
large enough $j$, $(\psi_j\circ \pi\circ h)_*$ is conjugate to $\rho_j\circ \rho^{-1}\circ h_*$.
\end{enumerate}
\end{lemma}

\begin{proof}{}
The existence of the subsequence $\{\rho_j\}$ is guaranteed by Canary-Epstein-Green
\cite[Thm. 3.1.4]{CEG}. The existence of the sub-manifolds  $\{Z_n\}$ and the comparison
maps $\{\psi_n\}$  with property (1) is given by \cite[Thm. 3.2.9 ]{CEG}. Properties (2) and (3) are obtained by
Brock-Canary-Minsky \cite[Lemma 2.8]{ELC2}. Property (4) is observed in \cite[Prop. 2.7]{ELC2}, see also
Anderson-Canary \cite[Lemma 7.2]{ACcores}.
\end{proof}

\section{The wrapping multicurve}
\label{wrap1}

In this section, we analyze how compact cores for algebraic limits immerse into geometric limits. We will
see that if $\{\rho_n\}\subset AH(S)$ converges algebraically to $\rho$ and $\{\rho_n(\pi_1(S))\}$ converges geometrically to $\hat\Gamma$,
then there is a level surface $F\subset N^0_\rho$ and a collection $Q$ of incompressible annuli in $F$ so
that the covering map \hbox{$\pi:N_\rho^0\to \hat N^0$} is an embedding on $F-Q$ and (non-trivially) wraps each component of
$Q$ around a toroidal component of $\partial\hat N^0_\rho$. The collection  $q$ of  core curves of $Q$ is called
the wrapping multicurve and we will define a wrapping number associated to each component of $q$ which
records how many times the surface wraps the associated annulus around the toroidal component of 
$\partial\hat N^0_\rho$.

\subsection{Wrapped surfaces}
\label{define wrapping numbers}

We first examine the topology of the situation.
Given a compact non-annular surface $G$ and $e\in\CC(G)$, let  $E=\collar(e)$ be an open collar neighborhood of $e$ on $G$,
$\hat G=G-E$, 
$$X= G \times [-1,1]\ \ \ {\rm and}\ \ \ \hat X= X - V\ \ \ {\rm where}\ \ \
V=  E \times (-\frac12, \frac12) \subset X$$
is a solid torus in the homotopy class of $e$. 
If $T=\partial V$ and \hbox{$\hat Z = \hat G\times\{0\} \cup T$}, then $\hat Z$ is a spine for $\hat X$.
An orientation on $G$ determines an orientation on $X$ and hence on $V$ which induces an orientation on $T$.
Let $m$ be an essential curve on $T$ that bounds a disk in $V$ and let $l$ be one of the components of 
$\partial\bar E \times \{0\}$. We orient this meridian and longitude so that the orientation of $(m,l)$ 
agrees with the orientation of $T$. We also decompose $T$ into two annuli with 
$$A = \partial \bar E \times [0,1/2] \cup E \times \{1/2\}\ \ \ {\rm and}\ \ \ \ B = \partial \bar E \times [-1/2,0] \cup E \times \{-1/2\}.$$

We will show that every map from $G$ to $\hat X$ that is homotopic, in $X$, to a level inclusion,
is homotopic, in $X$ to exactly one of a family $\{f_k:G \to \hat X\}_{k \in \IZ} $ of standard wrapping maps. 
Let $f_1:G \to X$ be an embedding such that the restriction of $f_1$ to $\hat G$ is $id \times \{0\}$, i.e.
$f_1(x)=(x,0)$ if $x\in \hat G$, and  $f_1|_{\bar E}$ is a homeomorphism to $A$.
For all $k\in\IZ$, let  $\phi_k : T \to T$ be an immersion which is the identity on $B$ and wraps $A$ ``k times around''
$T$, namely $(\phi_k)_*(m)=km$ and $(\phi_k)_*(l)=l$.
We then define $f_k :G \to \hat Z  \subset X$ by $f_k|_{\hat G} = f_1|_{\hat G}$ and $f_k|_{\bar E} = \phi_k \circ f_1|_{\bar E}$. 
Note that all of these maps are homotopic as maps to $X$. 
As maps to $\hat Z$ (or $\hat X$) they are homotopically distinct as can be seen by 
counting the algebraic intersection with a point on $A$ and a point on $B$. 
We will call $k$ the {\em wrapping number} of $f_k$.

The next lemma allows us to define a wrapping number for any map in the correct homotopy class.

\begin{lemma}{wrapping}
Let $g: (G, \partial G) \to (\hat X, \partial G\times [-1,1])$ be a map such that
$g$ is homotopic to $id\times \{0\}$, as a map into $(X,\partial G\times [-1,1])$.
Then there exists a unique $k \in \Z$ such that $g$ is homotopic to $f_k$ 
as a map into $(\hat X, \partial G\times [-1,1])$.
\end{lemma}

\begin{proof} 
Since $\hat Z$ is a spine of $\hat X$, we may assume that the image of $g$ lies in $\hat Z$. 
Since $g$ is homotopic to the level inclusion $id\times\{0\}$ on $\hat G$, we may homotope $g$ within
$\hat X$ so that $g$ agrees with the level inclusion on $\hat G$. Since every immersed incompressible annulus
in $\hat X$ with boundary in $T$ is homotopic, rel boundary, into $T$,
we can further homotope $g$, rel $\hat G$, such that $g(E) \subset T$. 
A simple exercise shows that any map of $E$ to $T$ that agrees with $id\times \{0\}$ on
$\partial E$ is homotopic to  the composition of $\phi_k$, for some $k$, and some power of a Dehn twist about $E$.
Since $g$ is homotopic to $id\times \{0\}$ within $X$, the Dehn twist is un-necessary, so
$g$ is homotopic to $f_k$ in $(\hat X,\partial G\times [-1,1])$.
\end{proof}

We recall that we will be considering the case where $\{\rho_n\}\subset AH(S)$ converges to $\rho$, $\{\rho_n(\pi_1(S))\}$
converges to $\hat\Gamma$,
and there is a level surface $F\subset N^0_\rho$ and a collection $E$ of incompressible annuli in $F$ so
that the covering map $\pi:N_\rho^0\to \hat N^0$ is an embedding on $F-E$ and (non-trivially) wraps each component of
$E$ around a toroidal component of $\partial\hat N^0_\rho$. We also have, for large enough $n$, a 2-bilipschitz map
\hbox{$\psi_n:\hat N \to N_{\rho_n}$} defined on a regular neighborhood of $\pi(F)$ so that each component of $\psi_n(\pi(E))$ bounds
a Margulis tube in $N_{\rho_n}$. The following lemma gives information about the image of a meridian  of a component of $\pi(E)$

\begin{lemma}{dehn twist}
Let $G$ be a compact surface, $e\in\CC(G)$ and let $\hat Z$ be the spine for $\hat X$ constructed above.
Suppose that  $\psi: \hat Z \to M$ is an embedding  into a 3-manifold $M$ such that
$\psi(T)$ bounds a solid torus $U$ disjoint from $\psi(\hat Z)$, and
$\psi(l)$ is homotopic to the core curve of $U$.

Then there exists $s\in\IZ$ such that
\begin{enumerate}
\item
$\psi(m + sl)$ bounds a disk in $U$,
\item
$\psi \circ f_0:G\to M$ is homotopic to $\psi \circ f_k \circ D^{ks}$ for all $k$, and
\item
$\psi \circ f_1$ is homotopic to $ \psi \circ f_k \circ D^{(k-1)s}$ for all $k$ 
\end{enumerate}
where $D:G\rightarrow G$ is a right Dehn twist about $E$.
\end{lemma}

\begin{proof}
Since $\psi(l)$ is homotopic to the core curve of $U$, it is a longitude for $U$. So, the meridian $m_U$ for $U$ will intersect $\psi(l)$ 
exactly once. Therefore, the pre-image $\psi^{-1}(m_U)$ of the meridian will intersect $l$ exactly once and must be of 
the form $m + sl$ for some $s\in \Z$. 

If $s=0$ then (2) and (3) hold,  since we may extend $\psi$ to an embedding $\bar\psi:X\to M$ and $f_0$ is homotopic to $f_k$
within $X$ for all $k$.

We now define a map $h:\hat Z \to \hat Z$  which allows us to reduce to the $s=0$ case.
Let $h$ be the identity on $\hat Z-A$ and let $h|_A=D_A^{-s}$ where $D_A$ is the right Dehn twist about the core curve of $A$ so that $ h_*(m)=m+ sl $.
Then $\psi\circ h:\hat Z\to M$ is an embedding so that $\psi\circ h(T)$ bounds $U$ and $\psi\circ h(m)$ bounds a disk in $U$.
Therefore, for any $k$, $\psi \circ h\circ f_0$ is homotopic to $\psi \circ h\circ  f_k$. The $f_k$-pre-image of $A$ in $G$ is a collection of $k$ parallel annuli and the map $h \circ f_k$ is equal to pre-composing $f_k$ with $s$ Dehn twists in each of the $k$ annuli. As $s$ Dehn twists in $k$ parallel annuli is homotopic to $ks$ Dehn twists in a single annulus we have that $h\circ f_k:G\to Z$ is homotopic to $f_k\circ D^{ks}$ for all $k$.
Properties (2) and (3) follow immediately.
\end{proof}

\subsection{Wrapping multicurves and wrapping numbers }
\label{wrap numbers}

In this section, we analyze how compact cores for algebraic limits immerse into geometric limits.
We identify the wrapping multicurve and produce a level surface
in the algebraic limit whose projection to the geometric limit is embedded off of a collar neighborhood
of the wrapping multicurve. At the end of the section, we define the wrapping numbers  of the wrapping
multicurves.

\begin{proposition}{limit set machine}
Suppose that $\{\rho_n\}\subset AH(S)$,
$\lim \rho_n=\rho$, and $\{\rho_n(\pi_1(S))\}$ converges geometrically to $\hat\Gamma$. Let
$\hat N=\Ht/\hat\Gamma$ and let $\pi:N_\rho\to \hat N$ be the obvious covering map. 
There exists a level surface $F$ in $N_\rho$, a multicurve $q=\{q_1,\ldots,q_r\}$ on $F$,  and
an open collar neighborhood $Q=\collar(q)\subset F$, so that
\begin{enumerate}
\item
$\pi$ restricts to an embedding on $F-Q$.
\item
$l_\rho(q)=0$ and if $Q_i$ is the component of $Q$  containing  $q_i$, 
then $\pi|_{Q_i}$ is an immersion, which is not an embedding, into  the boundary $T_i$ of a
cusp region $V_i$. 
\item
If $\hat J$ is a (closed) regular neighborhood of $\pi(F)$ in $\hat N^0$, then $\hat J$ is homeomorphic to
$F\times [-1,1]\setminus (Q\times (-\frac{1}{2},\frac{1}{2}))$ and $\partial_1J=\partial\hat J-\pi(\partial N_\rho^0)$ is 
incompressible in $\hat N^0$.
In particular, $\pi_1(\hat J)$ injects into $\hat\Gamma$.
\item \label{order}
If $d$ is a downward-pointing parabolic in $N_\rho$,
then \hbox{$i(d,q)=0$}. Moreover, if $d$ is not a component of $q$ and $c$ is
a curve  in $S$ which intersects $d$, then the geodesic representative $c^*$ lies above $d^*$ in $N_{\rho_n}$ for
all large enough $n$.

Analogously, if $d$ is an upward-pointing parabolic in $N_\rho$,
then \hbox{$i(d,q)=0$}. Moreover, if $d$ is not a component of $q$ and $c$ is
a curve  in $S$ which intersects $d$, then the geodesic representative $c^*$ lies below $d^*$ in $N_{\rho_n}$ for
all large enough $n$.
\item
If there is a compact core for $N_\rho$ which embeds, under $\pi$, in $\hat N$,
then $q$ is empty. 
\end{enumerate}
\end{proposition}

We will call $q$ the {\em wrapping multicurve} of the triple $(\{\rho_n\},\rho,\hat\Gamma)$.
We say that a parabolic curve $d$ for  $\rho$ is an {\em unwrapped parabolic} for the triple
$(\{\rho_n\},\rho,\hat\Gamma)$ if it does not lie in the wrapping multicurve $q$.

\begin{proof}{}
Let $\widehat{\mathcal H}$ be an invariant collection of horoballs for the parabolic elements of $\hat\Gamma$ and 
let $\mathcal{H}$ be the subset of $\widehat{\mathcal H}$ consisting of horoballs based at fixed points of parabolic
elements of  $\rho(\pi_1(S))$.
Let 
$$\hat N_\rho^0=({\bf H}^3-\bigcup_{H\in\widehat{\mathcal{H}}} H)/\hat\Gamma\ \ \ {\rm and}\ \ \ N_\rho^0=({\bf H}^3-\bigcup_{H\in\mathcal{H}} H)/\rho(\Gamma)$$
and let $(M,P)$ be a relative compact core for $N_\rho^0$.

Let $A$ be a maximal collection of disjoint, nonparallel essential annuli in $(M,P)$ with one boundary
component in $P$.
Since one may identify $M$ with $S\times [-1,1]$ so that $\partial S\times [-1,1]$  is identified
with a collection of components of $P$, one may identify $A$ with $a\times [-1,1]$
where $a=\{q_1,\ldots,q_t\}$ is a disjoint
collection of simple closed curves on $S$. Let $R$ be the complement in $S$ of a collar
neighborhood of the multicurve $a$.
Let $\{ R_j\}$ be the components of $R$ and let
$\Gamma^j=\rho(\pi_1(R_j))$. Notice that an element of $\Gamma^j$ is parabolic if and only if
it is conjugate to  an element of $\rho(\pi_1(\partial R_j))$. Proposition 6.4 in \cite{ELC2} implies that there exists
a proper embedding $h:R_j\to \hat N^0$ such that $h_*(\pi_1(R_j))$ is conjugate to $\Gamma^j$
for each $j$. In particular, $h(\partial R_j)\subset\partial \hat N^0$.

We now construct $F$. For each $j$, let $F_j$ be a lift of $h(R_j)$ to $N_\rho$.
For each $i$, let $Q_i$ be the annulus in $\partial N_\rho^0$ joining two components of
$\bigcup \partial F_j$ whose core curve is homotopic to $q_i$. Then $F=\bigcup F_j\cup \bigcup Q_i$
is a level surface for $N^0_\rho$.

We re-order $\{q_1,\ldots,q_r,q_{r+1},\ldots, q_t\}$ so that if $i\le r$, then $\pi|_{Q_i}$ is not
an embedding, while if $i>r$, then $\pi|_{Q_i}$ is an embedding. Let $q=\{q_1,\ldots,q_r\}$ and
$Q=Q_1\cup\cdots\cup Q_r$. Conditions (1)  and (2) are satisfied by construction.

Let $\hat J$ be a  (closed) regular neighborhood of $\pi(F)$ in $\hat N^0$.
By construction, $\hat J$ is homeomorphic to 
$F\times [-1,1]\setminus (Q\times (-\frac{1}{2},\frac{1}{2}))$. We first
prove that  
$$\partial_1 \hat J=\partial\hat J-\pi(\partial\hat N_\rho^1)\cong F\times\{-1,1\}$$
is incompressible in  $\hat N^0$.
Since $\partial_1 \hat J$ is clearly  incompressible in $\hat J$,
we only need to check that $\partial_1 \hat J$ is incompressible in $\hat N^0-{\rm int}(\hat J)$. 
Each component $E$ of $\partial_1\hat J$  is homeomorphic to $S$. If $E$ is not incompressible in 
$\hat N^0-\hat J$, then
there exists an embedded disk $D$ in $\hat N^0-{\rm int}(\hat J)$ which is bounded by a 
homotopically non-trivial curve in $\partial_1\hat J$.

By Lemma \ref{comparisons respect tubes}, there exists, for all large enough $n$, 
$Z_n$ and a \hbox{$2$-bilipschitz} embedding 
$\psi_n:Z_n\to N_{\rho_n}$ so that $\hat J\cup D\subset Z_n$ and if $T$ is a toroidal boundary component
of $\hat J$, then $\psi_n(T)$ bounds a Margulis tube  in $N_{\rho_n}$. 
Moreover, if $c$ is a curve in $R_j\cap T$, for some $j$,  then
$\psi_n(c)$ is homotopic to the core curve of the Margulis tube.
Let $J_n$ be the union of
$\psi_n(\hat J)$ and all the Margulis tubes bounded by toroidal components of $\psi_n(\partial \hat J)$.
Then, $J_n$ is homeomorphic to $F\times [0,1]$ and $F_n=\psi_n(\pi(F))$ is homotopic, within $J_n$,
to a level surface of $J_n$.
Moreover, Lemma \ref{comparisons respect tubes}(4) implies that  each level surface of $J_n$ is  properly
homotopic to a level surface in $N_{\rho_n}^1$ for all large enough $n$.
Hence, $\psi_n(D)$ is a disk in $N_{\rho_n}$ bounded by a homotopically non-trivial curve in 
an embedded incompressible surface, which  is impossible.
Therefore, $\partial_1\hat J$ is incompressible in $\hat N^0$.  Since $\partial\hat N^0$ is incompressible in $\hat N$, it follows
that $\pi_1(\hat J)$ injects into  $\hat\Gamma$.
We have established property (3).

We now turn to the proof of property (4).
Let $d$ be a parabolic curve for $N_\rho$. If $i(d,q)\ne 0$, then $\pi(d)$ is non-peripheral in
the regular neighborhood $\hat J$ of $\pi(F)$. Since, $\partial_1\hat J$ is incompressible in $\hat N^0$,
it follows that $\pi(d)$ is non-peripheral in $\hat N^0$. However, since $\pi(d)$ is
associated to a parabolic element of $\hat\Gamma$, this is impossible. Therefore, $i(d,q)=0$.

Now suppose that $d$ is an unwrapped downward-pointing parabolic. 
It remains to show that if $c$ is a curve on $S$ which intersects $d$, then the geodesic representative
of $c$ lies above the geodesic representative $d_n^*$ of $d$ in $N_{\rho_n}$ for all sufficiently large $n$.

We first observe that there exists an immersed annulus $A$ in $\hat N^0$ joining $\pi(d)$ to
an essential curve $a$ in the cusp region $V(\pi(d))$ associated to $\pi(d)$ in $\hat N$ whose interior
is disjoint from $\pi(F)$.
We may assume that $\hat J$ is disjoint from $V(\pi(d))$.
Since $\pi(d)$ is homotopic into $V(d)$, there exists
an essential curve $a$ in $\partial V(d)$ which is homotopic to $\pi(d)$.
Let $A$ be an immersed annulus in $\hat N^0$ joining $\pi(d)$ to $a$. 
If $A$ cannot be chosen so that its interior is disjoint from $\pi(F)$,
then there exists a curve $\pi(b)$ in $\pi(F-Q)$ which is homotopic
to $\pi(d)$ in $\hat N$, but $b$ is not homotopic to $d$ in $F$.
Then there exists $\gamma\in\hat\Gamma-\rho(\pi_1(S))$ such that $\gamma \rho(b)\gamma^{-1}=\rho(d)$, so
$\rho(b)$ is also parabolic.
Let $V(b)$ and $V(d)$ be
the distinct cusp regions associated to $b$ and $d$ in $N_\rho$. Since $\pi(b)$ is homotopic to
$\pi(d)$, $\pi(V(b))=\pi(V(d))$. 
Lemma \ref{comparisons respect tubes} implies that $\rho_n(b)$ is homotopic to $\rho_n(d)$ in $N_{\rho_n}$ for all
large enough $n$, which is a contradiction. Therefore, we may assume that the interior of $A$ is disjoint from $\pi(F)$ as claimed.

We next observe that $F_n=\psi_n(\pi(F))$ lies above $d_n^*$ for all large enough $n$. 
The annulus $A$ lifts to an annulus in $N_\rho$ which lies below $F$.  We may assume that,
for all large enough $n$, $A\subset Z_n$ and $\psi_n(a)$ is an essential curve in
the boundary of the Margulis tube associated to $d_n^*$ in $N_{\rho_n}$. Let $R$ be the component
of $F-Q$ containing $d$. Since $\ell_\rho(\partial R)=0$, a result of Otal \cite[Theorem A]{otal-curves}
implies that, for all large enough $n$, the geodesic representative of each component of $\partial R$ is unknotted
in $N_{\rho_n}$. Lemma 2.9 in \cite{BBCM} then implies that 
$\psi_n(\pi(R))$ is a level subsurface of $N_{\rho_n}^1$  for all large enough $n$.
Since $\psi_n(A)$ lies
below $\psi_n(\pi(R))$, $d_n^*$ lies below the embedded subsurface $\psi_n(\pi(R))$. 
Lemma 2.7 in \cite{BBCM} then implies that
$d_n^*$ also lies below $F_n$.

If $c$ is a curve on $S$ which intersects $d$ essentially, then
$c$ has a representative $c_n$ on $F_n$ of length at most $L(c)$, for all $n$, so
there exists a homotopy from $c_n$ to either
$c_n^*$ or to a Margulis region in $N_{\rho_n}$ associated to $c$ which has tracks of length at most
$D(c)$, where $D(c)$ depends only on $L(c)$ (see \cite[Lemma 2.6]{ELC2}).
If $d(\partial \MT_\epsilon^n,d_n^*)>D(c)$, then this homotopy will miss $d_n^*$, which implies that
$c_n^*$ lies above $d_n^*$. However, this will be the case if $\ell_{\rho_n}(d)$ is sufficiently close to $0$,
which occurs for all large enough $n$.

The proof of property (4)  for unwrapped upward-pointing parabolics is analogous.

If $q$ is non-empty and there is a compact core for $N$ which
embeds in $N_\rho$, then $\pi|_F$ is homotopic to an embedding. However, since 
$\partial_1\hat J$ has incompressible
boundary, this implies that $\pi|_F$ is homotopic to an embedding within $\hat J$, which is clearly impossible.
This establishes (5) and completes the proof.
\end{proof}

Let $q_i$ be a curve of $q$. We will now define the  wrapping number of $q_i$ with respect to $ ( \{ \rho_n \} , \rho,\Gamma ) $.
Consider the manifolds $X$ and $\hat X$ defined in section \ref{define wrapping numbers} with $G = F$ and $e=q_i$.
From (3) we get an inclusion  \hbox{$\iota:\hat J \to  \hat X$}. Furthermore,  $\iota\circ \pi:F\to \hat X$
is homotopic, as a map into $X$, to $id\times\{0\}$.
Lemma \ref{wrapping} implies that  there is a unique $k \in \Z$ such that $\iota\circ\pi$ is homotopic to $f_k$ as maps into $\hat X$.
We then define the {\em wrapping numbers} $w^+(q_i) = k$ and $w^-(q_i) = k-1$.
Of course, it is clear that $w^+$ determines $w^-$, but as we will see, it is convenient to keep track of both numbers. 
Notice that the parabolic corresponding to a curve $q_i$ is downward pointing (in $N_\rho$) if and only if $w^+(q_i)>0$.

If $q=\{q_1,\ldots,q_r\}$, then we get  $r$-tuples
$$w^+(q)=(w^+ (q_1),...,w^+ (q_r))\ \ \ {\rm and}\ \ \ \ w^-(q)=(w^-(q_1),...,w^-(q_r)).$$

\medskip\noindent
{\bf Remarks:} (1) Our wrapping numbers are closely related to the wrapping coefficients discussed in Brock-Canary-Minsky
\cite[Section 3.6]{ELC2} and the wrapping numbers defined by Evans-Holt \cite{evans-holt}.

(2) Proposition \ref{limit set machine} may be viewed as a special case (and amplification) of the analysis carried out in
section 4 of Anderson-Canary-McCullough \cite{ACM}. In the language of that paper, the subsurfaces $\{F_j\}$ are the 
relative compact carriers of the precisely embedded system $\{\rho(\pi_1(R_j)\}$ of generalized web subgroups.

\section{Asymptotic behavior of end invariants in convergent sequences}
\label{nec}

In this section, we prove that if a sequence of Kleinian surface groups converges, then some subsequence of
the end invariants bounds projections. Along the way we will see that the sequence of end invariants
also predicts the parabolics in the algebraic limit and whether they are upward-pointing or downward-pointing.
In combination with results from \cite{BBCM} we see that the asymptotic behavior of the end
invariants predicts all the lamination and curve components of the end invariants of the algebraic limit.
Predicting the conformal structures which arise is significantly more mysterious. We also see that the asymptotic
behavior of the end invariants predicts the wrapping multicurve and the associated wrapping numbers
in any geometric limit.

\begin{theorem}{necessity}
Suppose that $\{\rho_n\}$ is a  convergent sequence in $AH(S)$ with
end invariants $\{\nu_n^{\pm}\}$ and that $\lim\rho_n=\rho$. 
Then there exists a subsequence $\{\rho_j\}$ such that the sequence
$\{\nu_j^\pm\}$ bounds projections.

Furthermore, if $\{\rho_j\}$ is a subsequence such that
$\{\nu_j^\pm\}$ bounds projections, then
\begin{enumerate}
\item
$\ell_\rho(d)=0$ if and only if $d$ is a combinatorial parabolic for the sequence
$\{\nu_j^\pm\}$,
\item
A parabolic curve $d$ is upward-pointing in $N_\rho$ if and only if
$$|m(\nu_j^+,d,\mu)|-|m(\nu_j^-,d,\mu)|\to +\infty.$$
\item
A lamination $\lambda\in {\mathcal{EL}}(Y)$ is an ending lamination for an
upward-pointing (respectively downward-pointing) geometrically infinite end for $N_\rho$ if and only if
$\{\pi_Y(\nu_j^+)\}$ (respectively $\{\pi_Y(\nu_j^-)\}$)  converges to $\lambda\in{\mathcal{C}}(Y)\cup{\mathcal{EL}}(Y)$.
\item
If $\{\rho_j(\pi_1(S))\}$ converges geometrically to
$\hat\Gamma$, then the wrapping multicurve for
$(\{\rho_j\},\rho,\Gamma)$ is the collection of combinatorial wrapping parabolics
given by $\{\nu_j^\pm\}$ and if $d$ is a wrapping parabolic,
then the combinatorial wrapping number $w(d)$ agree with the actual
wrapping number $w^+(d)$.
\end{enumerate}
\end{theorem}

\noindent{\bf Remark:}
In general, it is necessary to pass to a subsequence since the phenomenon of
self-bumping  (see McMullen \cite{mcmullen} or Bromberg-Holt \cite{bromberg-holt})
assures that you can have a convergent sequence with one subsequence
where the wrapping multicurve is empty and another subsequence where the wrapping
multicurve is non-empty.

\begin{proof}
We first prove that any geodesic joining $\nu_n^+$ to $\nu_n^-$ always intersects some
bounded set.

\begin{lemma}{bounded condition}
Suppose that $\{\rho_n\}$ is a sequence in $AH(S)$ converging to $\rho$.
Let $\nu_n^{\pm}$ be the end invariants of $\rho_n$.
There is a bounded set $\mathcal{B}\subset\mathcal{C}(S)$ such that every geodesic
joining $\pi_S(\nu_n^+)$ to $\pi_S(\nu_n^-)$ intersects $\mathcal{B}$.
\end{lemma}

\begin{proof}
If $a$ is any curve in $\CC(S)$, then there is a uniform upper bound $L_a$
on the length $l_{\rho_n}(a)$ for all $n$.
It follows from Theorem \ref{bounded curves near geodesic}
that $a$ lies within \hbox{$D(L_a)=D_0(\max\{L_a,L_0\})$} of any geodesic joining
$\pi_S(\nu_n^+)$ to $\pi_S(\nu_n^-)$. 
One may thus choose $\mathcal{B}$ to
be a neighborhood of $a$ in $\mathcal{C}(S)$ of 
radius $D(L_a)$. 
\end{proof}

We next show that  $\ell_\rho(d)$ is non-zero if and only if 
$\{m(\nu_n^+,\mu,d)\}$ and $\{m(\nu_n^-,\mu,d)\}$ are  both eventually bounded
for any marking $\mu$.
This is a fairly immediate consequence  of  work of Minsky, namely Theorem \ref{kgcc fact} and
the Short Curve Theorem of \cite{ELC1}.

\begin{lemma}{d bounded}
Suppose that $\{\rho_n\}$ is a sequence in $AH(S)$ converging to $\rho$.
Let $\nu_n^{\pm}$ be the end invariants of $\rho_n$ and let $\mu$ be a
marking on $S$.
Then, $l_\rho(d)>0$ if and only if 
$\{m(\nu_n^+,d,\mu)\}$ and $\{m(\nu_n^-,d,\mu)\}$ are both eventually bounded. 
\end{lemma}

\begin{proof}
First suppose that $\{m(\nu_n^\beta,d,\mu)\}$ is not eventually bounded, then either
$\{{1\over l_{\nu_n^\beta}(d)}\}$ is not eventually bounded or
$\{\sup_{d\subset\partial Y}d_Y(\nu_n^\beta,\mu)\}$ is not eventually bounded.
If $\{{1\over l_{\nu_n^\beta}(d)}\}$ is not eventually bounded, then there exists a subsequence
$\{\rho_j\}$ of $\{\rho_n\}$ so that $\{\ell_{\rho_j}(d)\}$ converges to 0, which implies that
$\ell_\rho(d)=0$. If $\{\sup_{d\subset\partial Y}d_Y(\nu_n^\beta,\mu)\}$ is not eventually bounded,
then either 
\begin{enumerate}
\item
there exists a subsequence
for which $d$ is always a component of $\nu_{j}^\beta$, or
\item
there exists a sequence of subsurfaces $Y_{j}$ such that $d\subset \partial Y_{j}$ and
\hbox{$d_{Y_{j}}(\nu_{j}^\beta,\mu)\to\infty$}.
\end{enumerate}
In case (1),
$\ell_{\rho_{j}}(d)=0$ for
all $j$, so $\ell_\rho(d)=0$.
In case (2), since $\ell_{\rho_j}(\mu)$ is eventually
bounded and $\pi_{Y_{j}}(\mu_j^\beta)\in\overline{\pi_{Y_j}(\CC(\rho,L_B))}$ where $L_B$ is the Bers
constant for $S$ (see Brock-Bromberg-Canary-Minsky \cite[Section 2]{BBCM}), we see that
\hbox{$ {\rm diam}(\pi_{Y_{j}}(\CC(\rho,L))\to\infty$} for some $L$.
Theorem \ref{kgcc fact} then implies that $\lim\ell_{\rho_{j}}(d) = 0$, so again $\ell_\rho(d)=0$.
Therefore, in all cases, if $\{m(\nu_n^\beta,d,\mu)\}$ is not eventually bounded, then $\ell_\rho(d)=0$.
It follows that if $l_\rho(d)>0$ then
$\{m(\nu_n^+,d,\mu)\}$ and $\{m(\nu_n^-,d,\mu)\}$ are both eventually bounded. 

If $\ell_\rho(d)=0$, then Minsky's Short Curve Theorem \cite{ELC1} implies
that at least one of $\{\frac{1}{\ell_{\nu_n^+}(d)}\}$,  $\{\frac{1}{\ell_{\nu_n^-}(d)}\}$, and
$\{\sup_{d\subset\partial Y}d_Y(\nu_n^+,\nu_n^-)\}$ is not eventually bounded. (For a similar
restatement of the Short Curve Theorem in the quasifuchsian case see
Brock-Bromberg-Canary-Minsky \cite[Thm. 2.2]{BBCM0}.)
It follows that $\{m(\nu_n^\beta,d,\mu)\}$ is not eventually bounded for some $\beta\in\{\pm\}$.
\end{proof}

We now pass to a subsequence $\{\rho_j\}$  of $\{\rho_n\}$ so that
$\{\rho_j(\pi_1(M))\}$ converges geometrically to 
$\hat\Gamma$. 
Let $\hat N={\bf H}^3/\Gamma$. Let $F$, $q=\{q_1,\ldots,q_s\}$ and $Q$, be
the level surface, wrapping multicurve and collar neighborhood  of $q$ provided by
Proposition \ref{limit set machine}.
Let 
$$f:S\to F$$
be a homeomorphism such that $f_*=\rho\in AH(S)$
and let $\hat J$ be a closed regular neighborhood of $\pi(F)$ in $\hat N^0$.

The following lemma characterizes the asymptotic behavior of the end invariants
relative to an unwrapped parabolic.

\begin{lemma}{unwrapped behavior}
Suppose that $\{\rho_j\}$ is a sequence in $AH(S)$ converging to $\rho$
such that $\{\rho_j(\pi_1(S))\}$ converges geometrically to $\hat \Gamma$
and that $d$ is an unwrapped parabolic for the triple $(\{\rho_j\},\rho,\hat\Gamma)$.
Then,
\begin{enumerate}
\item
if $d$ is a downward-pointing cusp in $N_\rho$, then 
$\{ m(\nu_j^+,d,\mu)\}$ is eventually bounded, and
\item
if $d$ is an upward-pointing cusp in $N_\rho$,
then $\{ m(\nu_j^-,d,\mu)\}$ is eventually bounded.
\end{enumerate}
\end{lemma}

\begin{proof}
Let $d$ be a downward-pointing cusp in $N_\rho$ and let $L_\mu$ be an upper bound for the 
length of $\mu$ in $N_{\rho_j}$ for all $j$.
Let $a$ be a curve in $\mu$ which crosses $d$. Proposition \ref{limit set machine}
guarantees  that $a^*$ lies above $d^*$  in $N_{\rho_j}$ for all sufficiently large $j$. 
Theorem \ref{top is top} then implies that, for all sufficiently large $j$,
$d_Y(a,\nu_j^+)<D(L_\mu)$ if $Y\subset S$ is
any subsurface with $d$ in its boundary.
Therefore, $\{\sup_{d\subset\partial Y}d_Y(\nu_j^+,\mu\})$ is eventually bounded.

It remains to check that there is an eventual lower bound on $l_{\nu_j^+}(d)$.
(For a similar argument in the quasifuchsian case, see \cite[Lemma 2.5]{BBCM0}.)
Since $l_{\rho_j}(a)<L_\mu$ for all $j$,
there exists $\epsilon>0$ so that the geodesic representative $a^*_j$ of $a$ in
$N_{\rho_j}$ misses $\MT_\ep(d)$ for all $j$.
The convex core $\partial C(N_{\rho_j})$ of $N_{\rho_j}$ is the smallest convex submanifold of $N_{\rho_j}$ containing
all the closed geodesics. Epstein, Marden and Markovic \cite[Theorem 3.1]{EMM} proved that there is a
$2$-Lipschitz map $f_j:\partial_cN_{\rho_j}\to \partial C(N_{\rho_j})$ so that $f_j$ extends to a strong deformation
retraction of $\partial_cN_{\rho_j}\cup N_{\rho_j}$ onto $C(N_{\rho_j})$. In particular, if $R_j$ 
is a downward-pointing component of $\partial_cN_{\rho_j}$, then no closed geodesic in $N_{\rho_j}$ lies below
$f(R_j)$. If $l_{\nu_j^+}(d)=l_j<\ep/2$, then there is a representative $d_j$ of $d$ in the image $f_j(R_j)$ of
a downward-pointing component
$R_j$ of $\partial_cN_{\rho_j}$ which  has length at most $2l_j<\epsilon$, so is
contained in $\MT_\ep(d_j)$. Therefore, $a_j^*$ cannot intersect $d_j$, so is disjoint from $f(R_j)$. 
It follows that  $f(R_j)$ lies below $a_j^*$, which implies
that $d_j$ lies below $a_j^*$. Since $d_j$ is homotopic to $d_j^*$ within $\MT_\epsilon(d_j)$ and $a_j^*$ is disjoint
from $\MT_\epsilon(d_j)$,  we see that $d_j^*$ lies below $a_j^*$. However, this
contradicts Proposition \ref{limit set machine},
so $l_{\nu_j^+}(d)\ge \ep/2$ for all sufficiently large $j$ which completes the proof for
downward-pointing cusps.

The proof in the case that $d$ is an upward-pointing cusp is similar.
\end{proof}

The situation is more complicated for wrapped parabolics.

We will abuse notation by letting $q$ also denote the multicurve $f^{-1}(q)\subset S$ and by
letting $Q$ denote the subsurface $f^{-1}(Q)$ of $S$.
Let  $X= S \times [-1,1]$ and $\hat X = X - V$ where $V=  Q \times (-\frac12, \frac12) \subset X$ is a union of  open
solid tori in the homotopy class of $q$. Set 
$$\hat Z=(S-Q)\times\{0\}\cup\partial  V\subset \hat X.$$

If $q$ is a single curve, then we are in the situation of section \ref{define wrapping numbers} with $ G=S $ and $e=q$.
We encourage the reader to focus on this situation when first reading the section.

In general, if $q=\{q_1,\ldots,q_r\}$, we divide $S$ into a collection of overlapping subsurfaces 
$\{G_1,\ldots, G_r\}$ defined as follows: $G_i$ is the connected component of $S-(Q-Q_i)$ that contains $Q_i$. 
One may then divide $X$ up into overlapping submanifolds $\{X_1,\ldots, X_r\}$ where  $X_i=G_i\times [-1,1]$. 
Similarly, one may  divide $\hat X$ up into submanifolds 
$\{\hat X_1,\ldots, \hat X_r\}$ with $\hat X_i=X_i-V_i$ where $V_i=  Q_i \times (-\frac12, \frac12) \subset X_i$ .  
Let $T_i$ be the toroidal boundary component of $\hat X_i$.

Proposition \ref{limit set machine} implies that we may identify $\hat J$ with $\hat X$.
If \hbox{$k=(k_1,\ldots,k_r)$} we may define a map $f_k:S\to\hat J$ which agrees on each $G_i$ with the 
map \hbox{$f_{k_i}:G_i\to \hat X_i$} defined in section \ref{define wrapping numbers}.  
Each of the $f_k$ determines a representation $(f_k)_*$ in $AH(S)$.
Since $f$ is homotopic to $f_{w^+(q)}$, we see that $(f_{w^+(q)})_* = \rho$.

Given a component $Q_i $ of $Q$ we denote by $D_i : S\to S$ the right Dehn twist about $Q_i$. 
For an $r$-tuple \hbox{$k=(k_1,\ldots,k_r)$}, we set $D_q^k=D_1^{k_1}\circ \ldots \circ D_r^{k_r}$.

\begin{lemma}{wrapped behavior}
For all large enough $j$, there exists a $r$-tuple $s_j=(s_{1,j},\ldots,s_{r,j})$ such that 
$\rho_j^+= \rho_j \circ (D_q^{w^+(q) s_j})_*$ and $\rho_j ^-=\rho_j\circ (D_q^{w^-(q)s_j})_*$ have the following properties:

\begin{enumerate}
\item the sequences $ \{ \rho_j^+\}$ and $\{\rho_j^-\}$ converge in $AH(S)$ to $\rho^+$ and $\rho^-$.

\item If $q_i$ is a component of $q$, then $q_i$ is an upward pointing parabolic in $N_{\rho^-}$ and  a downward pointing pararabolic
in $N_{\rho^+}$ and is unwrapped in the triples $(\{\rho^+_j\}, \rho^+, \hat \Gamma)$ and $(\{\rho^-_j\}, \rho^-, \hat \Gamma)$.

\item For each $i$, $\lim |s_{i,j}|=+\infty$.
\end{enumerate}
\end{lemma}

\begin{proof}
For  all large enough $j$, there exists a 2-bilipschitz embedding \hbox{$\psi_j: \hat J \to N_{\rho_j}$} such that each component of 
$\psi_j \circ \phi(\partial V)$ bounds a Margulis tube in $N_{\rho_j}$ and $\psi_j(\hat J)$ is disjoint from the interior of these
tubes. Moreover, $(\psi_j\circ\hat f)_*$ is conjugate to $\rho_j$. In particular, if $l$ is the longitude 
of any component $T$ of $\partial V$,
then $\psi_j(l)$ is a longitude of the Margulis tube bounded by $\psi_j(T)$.

Given $j\in\IN$, Lemma \ref{dehn twist} applied to $G=G_i$ and $e=q_i$ implies that
for all $i$, there exists $s_{i,j}$ so that if $m_i$ and $l_i$ are the meridian and longitude
of $T_i$, then $\psi_j(m_i+s_{i,j}l_i)$ bounds a meridian of $\psi_j(T_i)$.
We set $f^-=f_{(0,...,0)}$ and $f^+=f_{(1,...,1)}$ and let \hbox{$\rho^+=f^+_*$}  and $\rho^-=f^-_*$.
Lemma \ref{dehn twist} implies that $\psi_j\circ f^+$ is homotopic to $ f \circ  D_q^{w^+(q) s_j}$
and that $\psi_j\circ f^+$ is homotopic to $ f \circ  D_q^{w^-(q) s_j}$. It follows that
$\{\rho_j^+\}$ converges to $\rho^+$ and that $\{\rho_j^-\}$ converges to $\rho^-$.
This establishes property (1) and property (2) is true by construction.

It remains to establish property (3). Notice that since $\lim\ell_{\rho_j}(q_i)=0$, the diameter of
the Margulis tube bounded by $\psi_j(T_i)$ is diverging to $+\infty$. It follows that the length
of the meridian of $\psi_j(T_i)$ diverges to $+\infty$. Since $\psi_j$ is 2-bilipschitz, there
is a uniform upper bound on the lengths of $\psi_j(l_i)$ and $\psi_j(m_i)$. Since the meridian
of $\psi_j(T_i)$ is homotopic to $\psi_j(m_i+s_{i,j}l_i)$, we must have $\lim |s_{i,j}|=+\infty$.
\end{proof}

We can now easily assemble the proof of Theorem \ref{necessity}. 
We first show that 
if $\{\rho_n\}$ converges, then there is a subsequence $\{\rho_j\}$ so that
$\{\nu_j^\pm\}$ bounds projections. We choose a subsequence so that $\{\rho_j(\pi_1(S))\}$
converges geometrically.
Lemma \ref{bounded condition} implies that $\{\nu_j^{\pm}\}$ satisfies
condition (a) of the definition of bounding projections. Lemma \ref{d bounded} implies that all the
curves $d$ which are not parabolic in the algebraic limit satisfy condition (b)(i). Lemma \ref{unwrapped
behavior} implies that if $d$ is an unwrapped parabolic, then it satisfies condition (b)(i), while Lemma
\ref{wrapped behavior} combined with Lemma \ref{unwrapped
behavior} implies that any wrapped parabolic curve $d$ satisfies condition (b)(ii). Therefore,
$\{\nu_j^\pm\}$ bounds projections as claimed.

We now suppose that $\{\rho_j\}$ is a subsequence so that $\{\rho_j(\pi_1(S))\}$ converges geometrically. 
Property (1) follows from Lemma \ref{d bounded}. Property (2) follows from Lemma \ref{unwrapped behavior}
if $d$ is an unwrapped parabolic. Property (2) for wrapped parabolics follows from Lemma \ref{unwrapped behavior}
and the facts, observed in section \ref{define wrapping numbers}, that $w^-(q)=w^+(q)-1$ and that
$d$ is upward pointing if and only if $w^+(q)$ is positive. 
Property (3) comes from Theorem \ref{limitend} (\cite[Theorem 1.1]{BBCM}). 
Property (4) follows from Lemma \ref{wrapped behavior}.

In general, if $\{\rho_j\}$ is a subsequence of $\{\rho_n\}$ so that $\{\nu_j^\pm\}$ bounds projections.
Then every subsequence of $\{\rho_j\}$ has a subsequence $\{\rho_k\}$ so that $\{\rho_k(\pi_1(S))\}$ converges
geometrically. Therefore, every subsequence of $\{\rho_j\}$ has a subsequence for which properties
(1)--(4) hold. It is then easily checked that properties (1)--(4)  hold for the original sequence $\{\rho_j\}$.
\end{proof} 

\section{Multicurves from end invariants}

\label{equiv}

In this section, we prove that if the sequence of end invariants bounds projections, then we
can find a sequence of pairs of  bounded length multicurves which bounds projections.

\begin{proposition}{equivalence}
Suppose that $\{\rho_n\}$ is a sequence in $AH(S)$ with end invariants $\{\nu_n^\pm\}$.
If $\{\nu_n^\pm\}$  bounds projections, then
there exists  a subsequence $\{\rho_j\}$ and a sequence of pairs of multicurves $\{c_j^\pm\}$
such that $\{\ell_{\rho_j}(c_j^+\cup c_j^-)\}$ is bounded and $\{c_j^\pm\}$ bounds projections.
\end{proposition}

The moral here is quite simple, although unpleasant technical difficulties arise in the actual proof.
If $\{\rho_n\}$ is a sequence of quasifuchsian groups, one might hope to be able to choose
$c_n^+$ and $c_n^-$ to be minimal length pants decompositions of the top and bottom conformal
boundaries of $N_n$. There are three technical issues that cause this simple algorithm to fail:
\begin{itemize}
\item The $c_n^+$ and $c_n^-$ cannot have curves in common.

\item A downward (upward) pointing unwrapped combinatorial parabolic cannot be in $c_n^+$ ($c_n^-$).

\item A wrapped combinatorial parabolic cannot be in either $c_n^+$ or $c_n^-$.
\end{itemize}
It is easy to construct examples where the minimal length pants decompositions fail to
satisfy any of these technical constraints. 
To deal with these issues, we will choose $c_n^+$ to be a minimal length pants decomposition
of $\nu_n^+$ which intersects any downward-pointing combinatorial parabolic, any combinatorial
wrapped parabolic and any ``sufficiently short'' curve on $\nu_n^-$. We then choose $c_n^-$ to
be a  minimal length pants decomposition
of $\nu_n^-$ which intersects any curve in $c_n^+$, any downward-pointing combinatorial parabolic,  and any combinatorial
wrapped parabolic.

In general, one might hope to choose $c_n^+$ to consist of a minimal length pants decomposition of each 
geometrically finite subsurface on the ``top,'' a curve for each upward-pointing parabolic
and a pants decomposition of each subsurface supporting an
upward-pointing geometrically infinite end which is ``close enough'' to the ending lamination. 
We will again need to be more careful in the actual proof.

\begin{proof}
We  first pass to a subsequence, still called $\{\rho_n\}$, so that if $d$ is a curve and $\beta\in\{\pm\}$,
then either $m(\nu_n^\beta,d,\mu)\to\infty$ or $\{m(\nu_n^\beta,d,\mu)\}$ is eventually bounded.
Let $b^\beta$ be the collection of curves  such that
$m(\nu_n^\beta,d,\mu)\to\infty$ if and only if  $d$  is in $b^\beta$. 
If $d$ lies in $b^+$ or $b^-$, then $d$ is a combinatorial parabolic, while if $d$ lies in both
$b^+$ and $b^-$, then $d$ is a combinatorial wrapped parabolic.

The following lemma implies that $b^+$ and $b^-$ are multicurves.

\begin{lemma}{triple app}
Suppose that $\{\rho_n\}$ is a sequence in $AH(S)$ with end invariants $\{\nu_n^\pm\}$ and
 $\{\nu_n^\pm\}$  bounds projections. If $d$ is either an upward-pointing or wrapped combinatorial parabolic 
and $c$ intersects $d$, then $\{m(\nu_n^+,c,\mu)\}$ is eventually bounded.

Similarly, if $d$ is either a downward-pointing or wrapped combinatorial parabolic 
and $c$ intersects $d$, then $\{m(\nu_n^-,c,\mu)\}$ is eventually bounded.
\end{lemma}

\begin{proof}
We give the proof in the case that $d$ is either an upward-pointing combinatorial parabolic or a
combinatorial wrapped parabolic, in which case \hbox{$m(\nu_n^+,d,\mu)\to\infty$}. The proof of the other
case is analogous.

First suppose that $\ell_{\nu_n^+}(d)\to 0$, so $l_{\nu_n^+}(c)\to\infty$ and $d$ is a curve 
in the base of the (generalized) marking $\mu(\nu_ n^+)$ (defined in section \ref{end invariants}) 
associated to $\nu_n^+$  for all large enough $n$.
In particular,  if $c\in\partial Z$, then 
$$d_Z(\mu,\nu_n^+)\leq d_Z(\mu,d)+ d_Z(d,\mu(\nu_ n^+))\le 2 i(\mu,d)+6.$$ 
(The second inequality follows from Lemma \ref{intersection bounds distance} and the fact that any two curves in
$\mu(\nu_n^+)$ intersect at most twice.)
Therefore,  if $\ell_{\nu_n^+}(d)\to 0$, then
$\{m(\nu_n^+,c,\mu)\}$ is eventually bounded.

Notice that, by reversing the roles of $c$ and $d$ in the previous sentence, we see that if $m(\nu_n^+,d,\mu)\to\infty$, 
then $\{\ell_{\nu_n^+}(c)\}$ is bounded away from zero.

So, we may suppose that  both $\{\ell_{\nu_n^+}(d)\}$ and $\{\ell_{\nu_n^+}(c)\}$ are  bounded away from zero,
and that \hbox{$\sup_{d\subset\partial Y}d_Y(\nu_n^+,\mu)\to\infty$}.  
Therefore, there exists a sequence of subsurfaces $Y_n$ with
\hbox{$d\subset \partial Y_n$}, so that \hbox{$d_{Y_n}(\nu_n^+,\mu)\to\infty$}. It follows that
\hbox{$d_{Y_n}(\nu_n^+,c)\to\infty$}.
Lemma \ref{inequality on triples} then implies
that if $Z$ is a subsurface with $c\in\partial Z$, then 
$$d_Z(\partial Y_n,\nu_n^+)\le  4$$
for all large enough $n$. So,
$$d_Z(\nu_n^+,\mu)\le d_Z(\partial Y_n,\nu_n^+)+d_Z(\partial Y_n,\mu)\le 4+d_Z(d,\mu)+1$$
for all large enough $n$. Since $d_Z(d,\mu)$ is bounded above by a function of $i(d,\mu)$,
$\{\sup_{c\subset\partial Z}d_Z(\nu_n^+,\mu)\}$ is eventually bounded.
Therefore, again $\{m(\nu_n^+,c,\mu)\}$ is eventually bounded.
 \end{proof}

We next claim that a curve cannot be ``short'' on both
the top and the bottom.

\begin{lemma}{uniform max}
If $\{\rho_n\}$ is a sequence in $AH(S)$ with end invariants $\{\nu_n^\pm\}$ and
 $\{\nu_n^\pm\}$  bounds projections, then
there exists $\delta_1>0$, so that if $d$ is any curve on $S$, then
$$\max\{\ell_{\nu_n^+}(d),\ell_{\nu_n^-}(d)\}>\delta_1.$$
\end{lemma}

\begin{proof}
If not, we may pass to a subsequence so that there exist curves $a_n$ so
that $\ell_{\nu_n^+}(a_n)+\ell_{\nu_n^-}(a_n)\to 0$. Then $a_n$ is a curve of $ \mu ( \nu_n^\pm)$ 
hence $d_Y ( a_n , \nu_n^\pm)\le 5 $ for any subsurface $Y$ that intersects $ a_n $ essentially.
If $\{a_n\}$ admits a constant subsequence $a$, then $m^{na}(\nu_n^+,a,\mu)\to\infty$ and $m^{na}(\nu_n^-,a,\mu)\to\infty$ 
which is not allowed by condition (b) of the definition of bounding projections. 
If not, by Lemma \ref{blow up subsurface},
there is a subsurface $Y$ such that, after taking a subsequence, \hbox{$ d_Y( \mu, a_n ) \to\infty$}.
Then we have \hbox{$ d_Y( \mu, \nu_n^\pm ) \to\infty$} 
and \hbox{$ d_Y (\nu_n^+,\nu_n^-)\le 5 $} which contradicts  both conditions (b)(i) and (b)(ii).
Therefore, no such subsequence can exist and we
obtain the desired inequality.
\end{proof}

We recall that the Collar Lemma (\cite[Theorem 4.4.6]{buser}) implies that any two closed geodesics of length at
most $2\sinh^{-1}(1)$ on any hyperbolic surface cannot intersect.
Let $e_n^\beta$ denote the multicurve on $S$ consisting of curves $d$ such that
$$\ell_{\nu_n^\beta}(d)<\min\{2\sinh^{-1}(1),\delta_1\}.$$

We now describe the construction of $c_n^\pm$ in the case that
$\{\rho_n\}$ is a sequence of quasifuchsian representations, so $\nu_n^\pm\subset{\mathcal{T}}(S)$ for all $n$.
Among the pants decompositions of $S$ which
cross every curve in \hbox{$b^-\cup e_n^-$}, choose one, $c_n^+$,  with minimal length in $\nu_n^+$ . 
Then among the pants decompositions of $S$ which cross every curve in \hbox{$b^+\cup c_n^+$}
choose one, $c_n^-$, with minimal length in $\nu_n^-$. We observe that the resulting sequences have bounded length.

\begin{lemma}{bounded length}
The sequences $\{l_{\rho_n}(c_n^+)\}$ and $\{l_{\rho_n}(c_n^-)\}$ are both bounded.
\end{lemma}

\begin{proof}
Notice that since $\{m^{na}(\nu_n^{-\beta},d,\mu)\}$ is bounded  for all $d\in b^\beta$ and $b^\beta$ has
finitely many components, there
exists $\delta_2>0$ such that if $d\in b^\beta$, then 
$$\ell_{\nu_n^{-\beta}}(d)>\delta_2.$$

Lemma \ref{uniform max} implies that if $d$ is a component of $e_n^\beta$,
then $\ell_{\nu_n^{-\beta}}(d)\ge \delta_1$.

Therefore, there is a lower bound, $\min\{\delta_2,\delta_1\}$, on the length,  in $\nu_n^+$, of every curve in
$b^-\cup e_n^-$. Since  $b^-\cup e_n^-$ contains a bounded number of curves, it is an easy exercise to
check that there is an upper bound on the length of a minimal length pants decomposition of $\nu_n^+$
intersecting  $b^-\cup e_n^-$, hence an upper bound on the length, in $\nu_n^+$, of $c_n^+$.

Since $c_n^+$ crosses every curve in $e_n^-$, every curve in $c_n^+$ has length, in $\nu_n^-$, at least $\min\{2\sinh^{-1}(1), \delta_1 \}$. 
Therefore, there is a lower bound,
$\min\{\delta_2,\delta_1,2\sinh^{-1}(1)\}$, on the length, in $\nu_n^-$, of every curve in  $c_n^+\cup b^+$.
It again follows that there is an upper bound on the length of $c_n^-$.

Bers \cite[Theorem 3]{bers-slice} proved that if $d$ is any curve on $S$, then
$$\ell_{\rho_n}(d)\le 2\ell_{\nu_n^\beta}(d)$$
for either $\beta=+$ or $\beta=-$. It follows that both 
$\{l_{\rho_n}(c_n^+)\}$ and $\{l_{\rho_n}(c_n^-)\}$ are bounded.
\end{proof}

Since $c_n^\beta$ and the base of the  marking $\mu(\nu_n^\beta)$ both have uniformly
bounded length in $\nu_n^\beta$,
there is a uniform upper bound on the intersection number between
$c_n^\beta $ and any base curve of the marking $\mu(\nu_n^\beta)$. 
Therefore, Lemma \ref{intersection bounds distance}
implies that there exists $K$ so that if
$Y\subseteq S$ is not a component of $\collar(c_n^\beta)$ or $\collar({\rm base}(\nu_n^\beta))$, then
\begin{equation}\label{near}
d_Y(c_n^\beta,\nu_n^\beta)\le K.
\end{equation}
If $Y$ is a component of $\collar({\rm base}(\nu_n^\beta))$ and $c_n^\beta$ crosses $Y$,
then, since $c_n^\beta$ has bounded length, there is a lower bound on the length of the core
curve of $Y$ and hence an upper bound on the length of the transversal  to $Y$ in the marking $\mu(\nu_n^\beta)$.
Again, this implies an upper bound on the intersection number between the transversal and $c_n^\beta$, so
inequality (\ref{near}) still holds.

Finally, we pass to a subsequence so that, for each $\beta$,
if $d$ is any curve then $d$  either lies in $c_n^\beta$ for all $n$ or for only
finitely many $n$. Since $c_n^\beta$ is a pants decomposition and $c_n^-$ crosses every curve in $ c_n^+ $, 
then for any curve $d$ there exists $\beta(d)\in\{\pm\}$ and $N(d)\in\IZ$ such such that $c_n^\beta$
crosses  $d$ for all $n\ge N(d)$.

The next lemma shows that the properties we have established suffice to show that $\{c_n^\pm\}$  bounds projections. 
We give the statement and the proof in the general case (i.e. $\rho_n$ is not assumed to be quasifuchsian).

\begin{lemma}{predicts}
Let $\{ \nu_n^\pm \} $ be a sequence of pairs of end invariants which bounds projections and 
let $\{c_n^\pm\}$ be a sequence of pairs of multicurves on $S$ such that
\begin{enumerate}
\item there exists $K' >0 $ such that $d_S(c_n^\beta,\nu_n^\beta)\le K' $,
\item 
there exists $K>0$ such that if $d\in\CC(S)$, then there exists $M(d)\in\N$ such that
if $Y\subset S$ with $d\subset\partial Y$,  $c_n^\beta$ crosses $d$,  then 
\begin{equation}\label{nearby}
d_Y(\nu_n^\beta,c_n^\beta)\le K,
\end{equation} 
for any $\beta\in\{\pm\}$ and any $n\ge M(d)$,
\item 
if $d$ is a wrapped combinatorial parabolic, then $c_n^\beta$ intersects $d$ 
for any \hbox{$\beta \in \{\pm\}$},
\item
if $d$ is an unwrapped downward (respectivley upward) pointing combinatorial parabolic,
then $c_n^+$ (resp. $c_n^-$) intersects $d$, and
\item
if $d$ is not a combinatorial parabolic, then there exists $\beta(d)\in\{\pm\}$ and 
$N(d)\in\N$ such that $c_n^{\beta(d)}$ crosses $d$ for all $n\ge N(d)$.
\end{enumerate}
Then $\{c_n^\pm\}$ bounds projections.
\end{lemma}

\begin{proof}
Since $\{\nu_n^\pm\}$ bounds projection, there exists a bounded set $\mathcal{B}$ so that any geodesic
joining $\pi_S(\nu_n^+)$ to $\pi_S(\nu_n^-)$ intersects $\mathcal{B}$. 
By property (1), $d_S(c_n^\beta,\nu_n^\beta)$ is uniformly bounded,
so the hyperbolicity of the curve complex implies that any geodesic joining $c_n^+$ to $c_n^-$ lies a bounded Hausdorff
distance from a geodesic joining $\pi_S(\nu_n^+)$ to $\pi_S(\nu_n^-)$, and hence lies a bounded distance from $\mathcal{B}$.
Therefore, any geodesic joining $c_n^+$ to $c_n^-$ intersects some bounded set ${\mathcal{B}}'$, so
$\{c_n^\pm\}$ satisfies condition (a) in the definition of bounding projections.

If $d$ is a combinatorial wrapped parabolic,
then $d$  crosses both $c_n^+$ and $c_n^-$ (by property (3)),
so inequality (\ref{nearby}) implies that  $d$ is a combinatorial wrapped parabolic
for $\{c_n^\pm\}$.

If $d$ is an unwrapped combinatorial parabolic, then there exists 
$\beta=\beta(d)$ so that $d\in b^{-\beta}$, so
$\{m(\nu_n^{\beta(d)},d,\mu)\}$ is eventually bounded and 
$d$ crosses $c_n^{\beta(d)}$ for all $n$ (by property (4).)
Inequality (\ref{nearby}) implies that $\{m(c_n^{\beta(d)},d,\mu)\}$ is eventually bounded, so
$d$ satisfies condition (b)(i).

If $d$ is not a combinatorial parabolic, then there exists $\beta=\beta(d)$ and $N(d)$ such
that $d$ crosses $c_n^\beta$ for all
$n\ge N(d)$ (by property (5)). Then, since
$\{m(\nu_n^{\beta},d,\mu)\}$ is eventually bounded,
inequality (\ref{nearby}) implies that $\{m(c_n^{\beta},d,\mu)\}$ is eventually bounded, so
$d$ satisfies condition (b)(i).
This completes the proof that condition (b) holds for every curve.
\end{proof}

In the quasifuchsian case,
Lemma \ref{bounded length}, inequality (\ref{near}) and  Lemma \ref{predicts}
imply that $\{c_n^\pm\}$ bounds projections, so we have completed
the proof of Proposition \ref{equivalence}
in the quasifuchsian case.

\medskip

We next suppose that there exists a subsequence  $\{\rho_n\}$ such that  for all $n$,
neither $\nu_n^+$ or $\nu_n^-$ is a lamination supported on all of $S$.
We list all the simple closed curves on $S$ by fixing a bijection $\alpha:\mathcal{C}(S)\to{\Bbb{N}}$. 

When choosing the $c_n^+$ on a subsurface $W$ that supports a conformal structure in $\nu^-_n$,
we will use a  procedure similar to the one used in the quasifuchsian case. If $W$ supports a lamination $\lambda$ in $\nu_n^+$,
we choose a pants decomposition that has bounded length and is ``close'' to $\lambda$, where close is taken to mean that
the curves in the pants decomposition lie above any short curve in $\nu_n^-$ and any of the first $n$ curves in
our list that overlap $W$. This will allow us to establish Properties (1)--(5) in Lemma \ref{predicts}. We now make this precise.

Let $c_n^+$ contain every simple
closed curve component of $\nu_n^+$. If $W $ is a subsurface which supports a conformal
structure in $\nu_n^+$, let $c_n^+|_{ W }$ be  a minimal length
pants decomposition of $ W $ which intersects every component of 
$b^-\cup e_n^-$ which overlaps $ W $.
If the subsurface $W $ is the support
of a lamination in $\nu_n^+$, let $c_n^+|_{ W }  $ be a pants decomposition of $ W $ of 
length at most $L_1$  in $N_{\rho_n}$, so that each curve in $c_n^+|_{ W }  $ lies above every
curve in $\alpha^{-1}([0,n])\cup e_n^-$ which overlaps $ W $ (see Lemma \ref{bounded curves above in ends} 
for the existence of such a pants decomposition). 

Similarly, we  define $c_n^-$ 
so that it contains every closed curve component of $\nu_n^-$. 
If $W $ is a subsurface which supports a conformal
structure in $\nu_n^-$, let $c_n^-|_{ W }$ be  a  minimal length
pants decomposition of $W$ which intersects every component of 
$b^+\cup c_n^+$ which overlaps $W$.
If the subsurface $W$ is the support
of a lamination in $\nu_n^-$, let $c_n^-|_{W}$ be a pants decomposition of $W$  of
length at most $L_1$ so that each curve in $c_n^-|_{W}$ lies below every
curve in $\alpha^{-1}([0,n])\cup c_n^+$ which overlaps $W$ (again see Lemma \ref{bounded curves above in ends}).\\

As in the quasifuchsian case, 
$\{\ell_{\rho_n}(c_n^+\cup c_n^-)\}$ is bounded and $\{c_n^\pm\}$ has properties (3), (4) and (5)  of Lemma \ref{predicts}.

Let $Y\subseteq S$ be an essential subsurface.
If $Y$ lies in a subsurface $W$ which supports a conformal structure in $\nu_n^\beta$. Then, as in  the proof of inequality (\ref{near}),
Lemma \ref{intersection bounds distance}  implies that
$$d_Y(\nu_n^{\beta}, c_n^\beta)\le K$$
for large enough $n$ as long as $Y$ is not a component of $\collar(c_n^\beta)$.
If a simple closed curve component $p$ of $ \nu_n^\beta $ intersects $Y$ essentially, 
then $ p\subset c_n^\beta$ and $ p$ is a closed curve without transversal in the base
of the generalized marking $\mu(\nu_n^\beta) $ associated to $\nu_n^\beta$ (see section \ref{end invariants}). Hence we have
$$d_Y(\nu_n^{\beta},c_n^\beta)\le 2.$$
Finally, if $Y$ overlaps a subsurface $W$ which is the base surface of a lamination component of $\nu_n^\beta$, 
and $n\ge \alpha(d)$ for some 
$d\subset \partial Y$ that intersects $W$ essentially, Theorem \ref{top is top} then implies that
$$d_Y(\nu_n^{\beta},c_n^\beta)\le D.$$
Notice that in this last case we need $ \partial Y \neq \emptyset$.
We have proved that $\{c_n^\pm\}$ satisfies property (2). 
Since $\nu_n^\beta$ is never an ending lamination supported on all of $S$, 
$\nu_n^\beta$  contains either a closed curve or a conformal structure,
so Property (1) holds as well. Lemma \ref{predicts} then allows us to complete the proof in the case that
$\nu_n^\beta$ is never an ending lamination supported on all of $S$.

\medskip

To complete the proof, we  consider the case where there exists $\beta_0\in\{\pm\}$ such that for all $n$,
$\nu_n^{\beta_0}$ is a lamination supported on all of $S$.
Notice that in this case,
Property (1) cannot hold, so we will need to again alter the construction somewhat.

If $\nu_n^\beta$ is not a lamination supported on all of $S$, then we choose $c_n^\beta$ exactly as above.
If $\nu_n^\beta$ is a lamination supported on all of $S$, then,
by Minsky's Lipschitz Model Theorem \cite{ELC1}, 
there exists $L_0$ and a tight geodesic $g_n$ joining $\mu(\nu_n^+)$ to $\mu(\nu_n^-)$ 
such that for any vertex $d$ of $g_n$, we have $\ell_{\rho_n}(d)\leq L_0$. 
Since $\{\nu_n^\pm\}$  bounds projections, 
there exists $K>0$ and a vertex $d_n$ of $g_n$, such that $d_S(d_n ,\mu)\leq K$.
Minsky's Lipschitz Model Theorem \cite{ELC1} again implies that there exists
a pants decomposition $c_n^\beta$ of $S$ containing a vertex of $g_n$ between $d_n$ and $\mu(\nu_n^\beta)$ such that
$\ell_{\rho_n}(c_n^\beta)\leq L_1$, and any curve in $c_n^\beta$ lies above 
every curve in $\alpha^{-1}([0,n])\cup e_n^-$ if $\beta =+$ and any curve in $c_n^\beta$ lies below every
curve in $\alpha^{-1}([0,n])\cup c_n^+$ if $\beta =-$. 
 
One then verifies properties (2)--(5)  of Lemma \ref{predicts} just as above. 
Property (1) was only used to prove condition (a), i.e. that every geodesic
in $\CC(S)$ joining $c_n^+$ to $c_n^-$ passes through a fixed bounded set.
However, in the case that $\nu_n^{\beta_0}$ is always a lamination supported on all of $S$,
it follows directly  from our construction and the hyperbolicity of the curve complex (\cite{masur-minsky}) that any geodesic
joining $c_n^+$ to $c_n^-$ passes within a uniformly bounded distance of $\mu$. This completes
the proof of Proposition \ref{equivalence} in our final case.
\end{proof}

\section{Bounded projections implies convergence}
\label{suff}

In this section we prove that if a sequence of Kleinian surface groups admits a pair of sequences of
multicurves of uniformly bounded length which bounds projections,
then it has a convergent subsequence. We first handle the case where the
sequence of multicurves does not have any combinatorial wrapped parabolics, and then
handle the general case by applying an argument motivated by work of
Kerckhoff and Thurston \cite{kerckhoff-thurston}.

\subsection{In the absence of combinatorial wrapped parabolics}

We recall that if a sequence $\{c_n^\pm\}$ of pairs of multicurves bounds projections 
and there are no combinatorial wrapped parabolics, then for any curve $d$ and complete marking $\mu$ there exists $\beta(d)$
such that $\{m(c_n^{\beta(d)},d,\mu)\}$ is eventually bounded.

\begin{proposition}{strong is sufficient}  
Suppose that $\{\rho_n\}$ is a sequence in $AH(S)$ and there exists a sequence $\{c_n^\pm\}$ of
pairs of multicurves such that  \hbox{$\{\ell_{\rho_n}(c_n^+\cup c_n^-)\}$} is bounded and 
$\{  c_n^\pm\}$  bounds projections and has no combinatorial wrapped parabolics.  
Then $\{\rho_n\}$ has a  convergent subsequence.
\end{proposition}

\noindent
{\bf Remark:} Notice that any bounded sequence in $QF(S)$ will admit  bounded length multicurves
which bound projections (any pair of filling pants decompositions will work).
Therefore, we can only conclude that  there exist a convergent subsequence.

Moreover, unlike in the end invariants case, a sequence of wrapped multicurves which bounds projections 
need not predict all the parabolics in the limit and need not predict which parabolics wrap.
Notice that if $\{\rho_n\}$ converges and $c^+$ and $c^-$ is any pair of filling multicurves, then the
constant sequence $\{c_n^\pm=c^\pm\}$ will be a sequence of pairs of bounded length multicurves
bounding projections. In this case, $\{c_n^\pm\}$ does not predict any parabolics or ending laminations.

\medskip

\begin{proof}
We first show that, after passing to a subsequence $\{\rho_j\}$, there exists a
fixed pants decomposition which has bounded length in all $N_{\rho_j}$.

\begin{lemma}{bounded pants}
Suppose that $\{\rho_n\}$ is a sequence in $AH(S)$ and consider
a sequence $\{c_n^\pm\}$ of pairs of multicurves which bound projections without combinatorial wrapped parabolics.
If $\{\ell_{\rho_n}(c_n^+\cup c_n^-)\}$ is bounded, then there exists a subsequence 
$\{\rho_j\}$  and a pants decomposition
$r$ of $S$, so that $\{\ell_{\rho_j}(r)\}$ is a bounded sequence.
\end{lemma}

\begin{proof}{}
By assumption, there is a bounded region ${\mathcal{B}}$ in ${\mathcal{C}}(S)$ such
that any geodesic joining $c_n^+$ to $c_n^-$ intersects $\mathcal{B}$. For all $n$, let $b_n$ be
a curve on the geodesic joining $c_n^+$ to $c_n^-$ which is contained in $\mathcal{B}$.
By Theorem \ref{bounded curves near geodesic}, there exists $D$ and $L$ such that, for all $n$, there
exists a curve $a_n\in{\mathcal{C}}(S)$ such that $d(a_n,b_n)\le D$ and
$\ell_{\rho_n}(a_n)\le L$.

If $\{a_n\}$ admits a constant subsequence, then we pass to the
appropriate subsequence of $\{\rho_n\}$ and the constant curve is
the first curve in our pants decomposition $r$.

If not, by Lemma \ref{blow up subsurface}, there is a subsurface $Y$ such that $d_Y(a_n,\mu)$ diverges. Since $a_n$ is contained in a bounded region of $\CC(S)$, $Y$ is a proper subsurface of $S$. By assumption, there exists
$\beta\in\{\pm\}$, so that $d_Y(c_n^\beta,\mu)$ is bounded, hence \hbox{$d_Y(c_n^\beta,a_n)\to\infty$}.
Then, by Theorem \ref{kgcc fact}, $\ell_{\rho_n}(\partial Y)\to 0$. In this case, the
components of $\partial Y$ are the first curves in $r$.

We now assume that $r$ is non-empty and not yet a pants
decomposition. We apply a mild variation of the above argument to
show that we can enlarge $r$. This will eventually complete the
proof.  Let $W$ be a component of $S-r$ which is not a
thrice-punctured sphere. Since
$r$ has uniformly bounded length, one may use Lemma \ref{bounded length pants}
to find, for all $n$, a curve $b_n\in\mathcal{C}(W)$ so that $\ell_{\rho_n}(b_n)$ is
uniformly bounded. By assumption, there exists $\beta\in  \{\pm\}$
so that $d_W(c_n^\beta,\mu)$ is eventually bounded. Let $L\ge L_0$ be an upper bound for both $\{\ell_{\rho_n}(c_n^\beta)\}$ and $\{\ell_{\rho_n}(b_n)\}$
(where $L_0=L_0(S)$ is the constant from Theorem \ref{bounded curves near geodesic}).
Theorem \ref{bounded curves near geodesic}  implies that there exists \hbox{$D=D(S,L)$} such that either
${\rm diam}(\pi_W({\mathcal{C}}(\rho_n,L))\le D$ or $d_W(c_n^\beta,{\mathcal{C}}(W,L,\rho_n))\le D$
for all $n$ (since $c_n^\beta\in{\mathcal{C}}(\rho,L)$).
In the first case, $d_W(b_n,c_n^\beta)\le D$, while in the second case there exists 
$a_n\in {\mathcal{C}}(W,L,\rho_n)$ such that $d_W(c_n^\beta,a_n)\le D$.
In the first case, we let $a_n=b_n$. Therefore, in either case,  we have constructed
a sequence $\{a_n\}$ in ${\mathcal{C}}(W)$ such that $\ell_{\rho_n}(a_n)\le L$ 
and $d_W(c_n^\beta,a_n)\le D$.

If $\{a_n\}$ admits a constant subsequence, then we pass to the
appropriate subsequence of $\{\rho_n\}$ and add the constant curve to $r$.
If not, by Lemma \ref{blow up subsurface} there is a subsurface $Y$ such that $d_Y(a_n,\mu)$ diverges. Since $\{d_W(a_n,\mu)\}$ is eventually bounded, $Y$ is a proper subsurface
of $W$.  We can again argue, as in the third paragraph of the proof, 
that \hbox{$d_Y(c_n^{\beta'},a_n)\to\infty$} for some $\beta'\in\{\pm\}$. By Theorem \ref{kgcc fact},
\hbox{$\ell_{\rho_n}(\partial Y)\to 0$}. In this case, we may add
$\partial Y-\partial W$ to $r$.
\end{proof}

Next we construct, for every curve in $r$ a transversal
which has bounded length in all $N_{\rho_j}$, perhaps after passage to a further subsequence.
By Lemma \ref{bounded length pants}, there are
bounded length pants decompositions $r_j^+$ and $r_j^-$ in $N_{\rho_j}$
containing $c_j^+$ and $c_j^-$, respectively.  We may pass to a
subsequence so that $r\cap r_j^+$ and $r\cap r_j^-$ are both constant.
(Here, we use $r\cap r_j^\beta$ as shorthand for the collection of curves which lie
in both $r$ and $r_j^\beta$.)

Let $d$ be a curve in $r$. There exists a choice of sign $\beta=\beta(d)\in\{\pm\}$
so that $m(c_j^\beta,d,\mu)$ is bounded for all $j$, perhaps after
again passing to a subsequence.  In particular, this implies that $d$ does not lie in $r_j^\beta$
(since $d$ must intersect $c_j^\beta$ if $m(c_j^\beta,d,\mu)$ is finite).
Let $G=G(d)$ be the subsurface of $S-(r\cap r_j^\beta)$ which contains $d$.

Let $H_j=H_j(d)$ be a hierarchy in $\mathcal{C}(G)$ joining $r_j^\beta\cap G$ and
$r\cap G$. Here we regard both $r_j^\beta\cap G$ and $r\cap G$ as markings without transversals.
(Hierarchies are defined and discussed extensively  in Masur-Minsky \cite{masur-minsky2}.)

Let $\sigma_j\in AH(G)$ be the unique Kleinian group so that $r_j^\beta\cap G$ is the collection of
upward-pointing parabolic and $r\cap G$ is the collection of downward-pointing parabolics.
Let $X_j=N_{\sigma_j}={\bf H}^3/\sigma_j(\pi_1(G))$.
(The hyperbolic manifold $X_j$ is called a maximal cusp, see Keen-Maskit-Series \cite{KMS} for a proof of
the existence and uniqueness of $X_j$. The existence also follows from Thurston's Geometrization Theorem
for pared manifolds, see Morgan \cite{morgan}.)
Notice that $r_j^\beta\cap G$ and $r\cap G$ are the end invariants of $X_j$.

Let $M_j$ be the model manifold associated to the hierarchy  $H_j$. (The construction of a model
manifold associated to a hierarchy is carried out in Minsky \cite[Sec. 8]{ELC1}.)
The Bilipschitz Model Manifold Theorem \cite{ELC2} guarantees that there exists a bilipschitz homeomorphism
\hbox{$g_j:M_j\to X_j$.}

The hierarchy $H_j$ is a family of tight geodesics. The base tight geodesic lies in ${\mathcal{C}}(G)$ and joins
$r_j^\beta\cap G$ to $r\cap G$. Theorem \ref{bowditch a priori} implies that there is a uniform upper bound
on the length $\ell_{\rho_j}(c)$ of any curve $c$ which is contained in a vertex of the base tight geodesic. 
Then $H_j$ is constructed iteratively by appending tight geodesics in curve complexes of subsurfaces of $G$ which
join vertices in previously added tight geodesics.
Since this process terminates after a finite (bounded) number of steps, Theorem \ref{bowditch a priori} implies
that there is a uniform upper bound on the length $\ell_{\rho_j}(c)$ of any curve $c$ contained in a vertex
in the hierarchy $H_j$.

The model manifold $M_j$
is constructed from {\em blocks}  of two isometry types, one homeomorphic to the product
of a one-holed torus and the interval and the other homeomorphic to the product of a four-holed sphere
and the interval, {\em tubes}, which are isometric to Margulis regions in hyperbolic 3-manifolds, and
a finite number of {\em boundary blocks}. Each block is associated to an edge of a geodesic in the curve complex of either a one-holed torus or a four-holed sphere. These geodesics are called 4-geodesics.

Let $\hat M_j$ be obtained from $M_j$ by removing the tubes and the boundary blocks. So, $\hat M_j$ consists
entirely of blocks. Since all the vertices have uniformly bounded length, the techniques of
section 10 of  Minsky \cite{ELC1}  (in particular,
see Steps 0--5)  imply that
there exists a $K$-Lipschitz map $h_j:\hat M_j\to N_{\rho_j}$ where $K$ depends only on $S$ and the uniform bound
on the lengths of the curves in $H_j$ obtained from Theorem \ref{bowditch a priori}.

Let $A_{d,j}$ be the intersection of $\hat M_j$ with $U(d)$, the tube in $M_j$ associated to $d$. The annulus
$A_{d,j}$ is made up of  $s_j(d)+1$ bounded geometry annuli where $s_j(d)$ is the number of edges of 4-geodesics
in $H_j$ whose domains contain $d$ in their boundary. The arguments in Theorem 9.11 of Minsky \cite{ELC1}
imply that 
$$s_j(d)\le C\left(\sup_{d\in\partial Y, Y\ne\collar(d)} d_Y(r,r_j^\beta)\right)^a$$
for uniform constants $C$ and $a$.
However, 
$$\sup_{d\in\partial Y, Y\ne\collar(d)} d_Y(r,r_j^\beta)\le m(c_j^\beta,d,\mu)+
\sup_{d\in\partial Y, Y\ne\collar(d)} d_Y(r,\mu).$$
The first term on the right hand side is uniformly bounded by assumption, while the second
term is finite and independent of $j$. Therefore, $s_j(d)$ is bounded, which implies that the geometry of $A_{d,j}$ is
uniformly bounded.

It follows that there is an essential curve $t_{d,j}$  of uniformly bounded length 
in  $\partial \hat M_j$  which is disjoint from the
boundaries of the annuli associated to components of $r\cap G-d$ and intersects $U(d)$ minimally,
i.e. in two arcs if $U(d)$  separates the component of  $G-(r\cap G)$ it is contained in and in one arc otherwise.
The image $g_j(t_{d,j})$ in $X_j$ is a curve, of uniformly bounded length, which lies above the cusp
associated to $d$. Theorem \ref{top is top} then implies that
$d_Y(t_{j,d},r_j^\beta)$ is uniformly bounded when $d\subset\partial Y$. Since,
$m(c_j^\beta,\mu)$ is uniformly bounded and
$$|d_Y(c_j^\beta,\mu)-d_Y(r_j^\beta,\mu)|\le 1,$$
we see that $d_Y(t_{j,d},\mu)$ is uniformly bounded for any subsurface $Y\subset S$ whose boundary contains $d$.
Since any two curves which are disjoint from $r\cap G-d$ and intersect $d$ minimally
differ, up to homotopy, by a power of a Dehn twist in $U(d)$, there are only finitely many possibilities for $t_{j,d}$.
Therefore, we
may pass to a subsequence so that $t_{j,d} = t_d$ for a fixed curve
$t_d$.  The length $\ell_{\rho_j}(t_d)$ is uniformly bounded, since
$h_j(t_d)$ is a bounded length representative of $t_d$ in $N_{\rho_j}$.

We have found a pants decomposition $r$ and a system of transversals
$\{t_d\}_{d\in r}$ such that all curves in $r$ and their transversals
have uniformly bounded length in $\{N_{\rho_j}\}$. It then follows from Thurston's Double Limit Theorem
\cite{thurston2,otal-double}
that $\{\rho_j\}$ has a convergent subsequence.
\end{proof}

\noindent
{\bf Remark:} With a little more care, one may use this same argument to find a surface in $N_{\rho_j}$, for all large
enough $j$, where $r$ and $\{t_d\}_{d\in r}$ have uniformly bounded length. One can then verify convergence up to
subsequence more directly.

\subsection{The general case}
We now use ideas based on work of Kerckhoff and Thurston \cite{kerckhoff-thurston} to handle the general
case.

\begin{proposition}{sufficient}  
Suppose that $\{\rho_n\}$ is a sequence in
$AH(S)$ and there exists a sequence of pairs, $\{c_n^\pm\}$, of multicurves such that
$\{\ell_{\rho_n}(c_n^+\cup c_n^-)\}$ is bounded and $\{ c_n^\pm\}$ bounds projections.
Then $\{\rho_n\}$ has a convergent subsequence.
\end{proposition}

\begin{proof}
Let $q$ be the set of combinatorial wrapped parabolics for $\{c_n^\pm\}$. We recall that $d\in q$ if and only if
$\{m^{na}(c_n^+,d,\mu)\}$ and  $\{m^{na}(c_n^-,d,\mu)\}$  are both
eventually bounded  and there exists $w=w(d)\in \mathbf{Z}$ and a sequence 
\hbox{$\{s_n=s_n(d)\}\subset\Z$} such that  $\lim |s_n|=\infty$ and both
\hbox{$\{d_Y((D_Y^{s_nw}(c_n^+),\mu)\}$} and \hbox{$\{d_Y(D_Y^{s_n(w-1)}(c_n^-),\mu)\}$} are
eventually bounded when $Y=\collar(d)$.

Notice that if $q$ is empty, then  Proposition \ref{sufficient} follows from
Proposition \ref{strong is sufficient}. We first observe that $q$ is a multicurve.

\begin{lemma}{}
The set $q$ of combinatorial wrapping parabolics is a multicurve.
\end{lemma}

\begin{proof}
Suppose that $q$ contains intersecting curves $c$ and $d$, and let $Y=\collar(c)$ and $Z=\collar(d)$.
Lemma \ref{inequality on triples} then implies that
$$\min\{d_Y(\partial Z,c_n^+),d_Z(\partial Y,c_n^+)\}\le 10$$
which contradicts the fact that both
$d_Y(c_n^+,\mu)\to\infty$ and \hbox{$d_Z(c_n^+,\mu)\to\infty$}.
\end{proof}

Let $Q=\bigcup_{q_i\in q} Q_i=\collar(q_i)$ be a regular neighborhood of $q$ and consider the diffeomorphisms 
$$\Phi^+_n=\Pi_{q_i\in q} D_{Q_i}^{s_n(q_i) w(q_i)}\ \ {\rm and}\ \  \Phi^-_n=\Pi_{q_i\in q } D_{Q_i}^{s_n(q_i) (w(q_i)-1)}$$
where $D_{Q_i}$ is the right Dehn twist about the annulus $Q_i$.

\begin{lemma}{unwrapped}
The pairs of sequences $\{\Phi^+_n(c_n^\pm)\}$ and $\{\Phi^-_n(c_n^\pm)\}$
both bound projections and have no combinatorial wrapped parabolics.
\end{lemma}

\begin{proof}
We first prove that $\{\Phi^+_n(c_n^\pm)\}$  bounds projections. 
 
Let $d$ be a curve in $q$.  Since $\{c_n^\pm\}$ bounds projections, $d$ lies a uniformly bounded distance
from any geodesic joining $c_n^+$ to $c_n^-$. Notice that if $c\in\mathcal{C}(S)$, then
$d_S(d,\Phi^+_n(c))=d_S(d,c)$. Since any geodesic joining $\Phi_n^+(c_n^+)$ to $\Phi_n^+(c_n^-)$ is the image
under $\Phi_n^+$ of a geodesic joining $c_n^+$ to $c_n^-$, it follows that $d$ also lies a uniformly bounded distance
from any geodesic joining $\Phi_n^+(c_n^+)$ to $\Phi_n^+(c_n^-)$. Hence the pair
of sequence $\{\Phi_n^+(c_n^+)\}$ and $\{\Phi_n^+(c_n^-)\}$ 
satisfies condition (a) in the definition of bounding projections.

Let $d\subset S$ be a simple closed curve which is not a component of $q$. If $d$ does not cross $q$ then 
$m ( c_n^\pm, d, \mu ) = m ( \Phi_n^+( c_n^\pm), d, \mu ) $ for all $n$. Since $\{c_n^\pm\}$ bounds projections
and $d$ is not a combinatorial wrapping parabolic, it follows that there exists 
$\beta\in\{\pm\}$ such that $\{m ( \Phi_n^+( c_n^\beta), d, \mu)\}$ is eventually bounded.

If $d$ crosses a component $ q_i $ of $q$, it follows from the definition of 
$\Phi_n^\pm$ that $d_{Q_i}(d,\Phi_n^+(c_n^-))\longrightarrow \infty$ where $Q_i$
is the collar neighborhood of $q_i$.
Lemma \ref{inequality on triples}  then implies that if $n$ is large enough, then
$ d_Y(q_i, \Phi_n^+(c_n^-) )\leq 4 $ for any 
subsurface $Y$ whose boundary contains $d$. 
Thus, again if $n$ is large enough, by Lemma \ref{intersection bounds distance},
$$d_Y(\mu,\Phi_n^+(c_n^-))\le d_Y(\mu,q_i)+  d_Y(q_i, \Phi_n^+(c_n^-) )\le 1+2i(q_i,\mu)+4=5+2i(q_i,\mu)$$
for any 
subsurface $Y$ whose boundary contains $d$. 
Therefore, $\{m ( \Phi_n^+( c_n^-), d, \mu)\}$ is eventually bounded.

If $d=q_i$ is a component of $Q$, then $m^{na}(c_n^+, q_i, \mu ) = m^{na}( \Phi_n^+( c_n^+), q_i, \mu ) $ for all $n$,
so  $\{m^{na} ( \Phi_n^+( c_n^+), q_i, \mu)\}$ is eventually bounded. By definition of $\Phi_n^+$,
$\{d_{Q_i}(\Phi_n^+(c_n^+),\mu)$ is eventually bounded. Therefore,
$\{m( \Phi_n^+( c_n^+), q_i, \mu)\}$ is eventually bounded

We have proved that for any simple closed curve $d \subset S $
there is $ \beta $ such that $m ( \Phi_n^+( c_n^\beta), d, \mu ) $ is eventually bounded. 
This completes the proof that the pair
$\{\Phi_n^+(c_n^\pm)\}$ bounds projections without combinatorial wrapped parabolics.

The proof that the sequences of pairs $\{\Phi_n^-(c_n^\pm)\}$ bounds projections without combinatorial wrapped parabolics is analogous.
\end{proof}

For each $n$, consider the representations
$$\rho_n^+=\rho_n\circ (\Phi_n^+)_*^{-1} \ \ \ {\rm and}\ \ \ \rho_n^-=\rho_n\circ (\Phi_n^-)_*^{-1}.$$
By construction, the sequences $\{\ell_{\rho_n^\beta}(\Phi_n^\beta(c_n^\pm))\}=\{\ell_{\rho_n}(c_n^\pm)\}$ are
uniformly bounded for any $\beta\in\{\pm\}$. Lemma \ref{unwrapped} implies that $\{\Phi^+_n(c_n^\pm)\}$ and 
$\{\Phi^-_n(c_n^\pm)\}$ both bound projections  and have no combinatorial wrapped parabolics,
so Proposition \ref{strong is sufficient} implies that we may pass to a subsequence so that
both $\{\rho_n^+\}$ and $\{\rho_n^-\}$ converge to discrete, faithful representations $\rho^+$ and
$\rho^-$.

Extend $q$ to a pants decomposition $p$ of $S$. If $d\in p$, then
$\ell_{\rho_n}(d)=\ell_{\rho_n^+}(d)$ for all $n$, so $\{\ell_{\rho_n}(d)\}$ is bounded.
Let  $\hat p$ be a maximal collection of transversals to the elements of $p$
(i.e. each element of $\hat p$ intersects exactly one element of $p$ and does so minimally).
If $t\in\hat p$ is a transversal to an element of $p-q$, then again
$\ell_{\rho_n}(t)=\ell_{\rho_n^+}(t)$ for all $n$, so $\{\ell_{\rho_n}(t)\}$ is bounded

\begin{lemma}{}
If $t\in\hat p$ is a transversal to an element $d$ of $q$, then
$\{\ell_{\rho_n}(t)\}$ is bounded.
\end{lemma}

\begin{proof}{}
We show that any subsequence  of $\{ \rho_n\}$ contains a further subsequence such that $\{\rho_n(t)\}$ converges. 
Our result then follows immediately.

We first pass to a subsequence, and fix a specific representative in each conjugacy class, so that  
$\{\rho_n^+=\rho_n\circ (\Phi_n^+)_*^{-1}\}$
converges as a sequence of representations into ${\rm PSL}(2,\mathbb C)$. (The
existence of such a subsequence follows from Lemma \ref{unwrapped} and Proposition \ref{strong is sufficient}.)
Since $\Phi_n^+$ and $\Phi_n^-$ restrict to the identity on $S-Q$, and $\{\rho_n^-\}$ has a convergent
subsequence in $AH(S)$ (again by Lemma \ref{unwrapped} and Proposition \ref{strong is sufficient}), we may
pass to a further subsequence so that  $\{\rho_n^-\}$ also converges as a sequence of representations 
into ${\rm PSL}(2,\mathbb C)$.

Let us first consider the case where $t$ intersects $d$ exactly once. 
Then, with an appropriate choice of basepoint for $\pi_1(S)$, we have
$$\rho_n^-(t)=\rho_n(d^{(w(d)-1)s_n}t)=\rho_n^+(d^{-s_n}t),$$ 
so $\rho_n^+(d^{-s_n})=\rho_n^-(t)\rho_n^+(t)^{-1}$. Since $\{\rho_n^-(t)\}$ and $\{\rho_n^+(t)\}$ both converge 
we immediately conclude that \hbox{$\{\rho_n^+(d^{s_n})=\rho_n(d^{s_n})\}$} and 
\hbox{$\{\rho_n(t)=\rho_n(d^{-w(d)s_n})\rho_n^+(t)\}$} converge.

In the slightly more complicated second case where $t$ intersects $d$ twice, we argue by contradiction. 
We first homotope $t$ so  that the two points of $t\cap d$ coincide. Then $t$ is the concatenation of two loops $a$ and $b$ which are 
freely homotopic to curves that are disjoint from $d$ and $\rho_ n(t)=\rho_n(a b)$. With an appropriate choice of basepoint for $\pi_1(S)$, 
we have
$$\rho_n(a)=\rho_n^+(a)=\rho_n^-(a)\textrm{, }\quad \rho_n(d)=\rho_n^+(d)=\rho_n^-(d),$$
and
$$\quad\rho_n^-(b)=\rho_n(d^{(w(d)-1)s_n}bd^{-(w(d)-1)s_n})=\rho_n^+(d^{-s_n}bd^{s_n}).$$

Suppose that  $\{\rho_n(d^{s_n})=\rho_n^+(d^{s_n}\}$ exits every compact subset of ${\rm PSL}(2,\mathbb C)$ and pick \hbox{$p\in\mathbb H^3$}. 
Since the fixed points of $\rho_n^+(d)$ and $\rho_n^+(b)$ converge to
distinct sets (i.e. the fixed points of $\rho^+(d)$ and $\rho^+(b)$), $\rho_n^+(d^{s_n})(p)$ converges
to a point in $\partial\mathbb H^3$ disjoint from the fixed point set of $\rho^+(b)$. It follows that
$$d(\rho_n^+(bd^{s_n})(p),\rho_n^+(d^{s_n})(p))\to\infty.$$
Applying $\rho_n^+(d^{-s_n})$ to each term
we see that 
$$d(\rho_n^+(d^{-s_n}bd^{s_n})(p),p)\to\infty,$$
which contradicts the fact that \hbox{$\{\rho_n^-(b)=\rho_n^+(d^{-s_n}bd^{s_n})\}$} converges. 
Therefore,  a subsequence of $\{\rho_n(d^{s_n})\}$ converges. 
It follows that,  with the same subsequence,
\hbox{$\{\rho_n(b)=\rho_n(d^{-w(d)s_n})\rho_n^+(b)\rho_n(d^{w(d)s_n})\}$} and \hbox{$\{\rho_n(t)=\rho_n(ab)\}$} both converge.
(For a related argument see Anderson-Lecuire \cite[Claim 7.1]{eggshell}.) This completes the proof.
\end{proof}

We have exhibited a pants decomposition and a complete collection of
transversals all of whose images under $\rho_n$ have bounded
length. Therefore,  Thurston's Double Limit Theorem
\cite{thurston2,otal-double} again implies that $\{\rho_n\}$ has a convergent subsequence. 
\end{proof}

\section{Conclusion}
\label{conc}

We will now assemble the previous results to establish Theorems \ref{main theorem}, \ref{predictive power} and \ref{multi-curve version}. 
Let $S$ be a compact, orientable surface and let $\{\rho_n\}$ be a sequence in $AH(S)$ with end invariants $\{\nu_n^\pm\}$. 

\medskip\noindent
{\em Proof of Theorem \ref{main theorem}:}
If $\{\nu_n^\pm\}$  has a subsequence $\{\nu_j^\pm\}$ which bounds projections, 
then Proposition \ref{equivalence} implies that there exists  a further subsequence, still called $\{\rho_j\}$, 
and a sequence $\{c_j^\pm\}$ of pairs of multi-curves  such that $\{\ell_{\rho_j}(c_j^+\cup c_j^-)\}$ is  bounded and
$\{c_j^\pm\}$ bounds projections. Theorem \ref{sufficient} then implies that $\{\rho_j\}$, and hence $\{\rho_n\}$,
has a convergent subsequence. On the other hand, 
if $\{\rho_n\}$ has a convergent sequence, it follows immediately from Theorem \ref{necessity} that some subsequence of
$\{\nu_n^\pm\}$ bounds projections.
\qed

\medskip

Theorem \ref{predictive power} is precisely the second part of Theorem \ref{necessity}.

\medskip\noindent
{\em Proof of Theorem \ref{multi-curve version}:}
Theorem \ref{sufficient} implies that if there exists a sequence $\{c_n^\pm\}$ of pairs of multi-curves  such 
that $\{\ell_{\rho_n}(c_n^+\cup c_n^-)\}$ is  bounded and
$\{c_n^\pm\}$ bounds projections, then $\{\rho_n\}$ has a convergent subsequence.
On the other hand, if $\{\rho_n\}$ has a convergent subsequence $\{\rho_j\}$, then we may simply pick any filling pair $c^\pm$ of
multi-curves and set $c_j^\pm=c^\pm$ for all $j$. Then, since $\{\rho_j\}$ is convergent,
\hbox{$\{\ell_{\rho_j}(c_j^+\cup c_j^-)\}$} is  bounded and $\{c_j^\pm\}$ bounds projections
\qed

\end{document}